\documentclass[11pt]{amsart}
\usepackage{amsmath}
\usepackage{latexsym, bm}
\usepackage{indentfirst}
\usepackage{geometry}
\usepackage{array}
\usepackage{lipsum}
\usepackage{setspace}
\usepackage{multirow}
\usepackage{tabu}
\usepackage{amsfonts}
\usepackage{amsthm}
\usepackage{fancyhdr}
\usepackage{tikz}
\usepackage{amscd}
\usepackage{amssymb}
\usepackage{enumerate}

\usepackage{amssymb,amsfonts}
\usepackage[all,arc]{xy}
\usepackage{enumerate}
\usepackage{mathrsfs}
\usepackage{geometry}
\usepackage{amsmath}
\usepackage{latexsym, bm}
\usepackage{indentfirst}
\usepackage{CJK}
\usepackage{amsthm}
\usepackage{hyperref}
\usepackage{cite}

\usepackage{amsmath}
\usepackage{latexsym, bm}
\usepackage{indentfirst}
\usepackage{geometry}
\usepackage{array}
\usepackage{lipsum}
\usepackage{setspace}
\usepackage{multirow}
\usepackage{tabu}
\usepackage{amsfonts}
\usepackage{amsthm}
\usepackage{fancyhdr}
\usepackage{tikz}
\usepackage{amscd}
\usepackage{amssymb}

\usepackage{tikz-cd}
\usepackage{extarrows}

\usepackage{amssymb,amsfonts}
\usepackage[all,arc]{xy}
\usepackage{enumerate}
\usepackage{mathrsfs}
\usepackage{geometry}
\usepackage{amsmath}
\usepackage{latexsym, bm}
\usepackage{indentfirst}
\usepackage{CJK}
\usepackage{amsthm}
\usepackage{hyperref}
\usepackage{cite}

\def \f{\frac}
\def \a{\alpha}
\def \b{\beta}

\def \lra{\longrightarrow}

\def \p{\partial}

\def \G{\Gamma}

\newtheorem{thm}{Theorem}[section]
\newtheorem{cor}[thm]{Corollary}
\newtheorem{lem}[thm]{Lemma}
\newtheorem{prop}[thm]{Proposition}

\newtheorem{defn}[thm]{Definition}
\newtheorem{remk}[thm]{Remark}

\newtheorem{que}[thm]{Question}

\DeclareMathOperator{\C}{\mathbb C}

\begin{document}
\title{Kodaira dimensions of almost complex manifolds I}
\author{Haojie Chen }
\address{Department of Mathematics\\ Zhejiang Normal University\\ Jinhua Zhejiang, 321004, China}
\email{chj@zjnu.edu.cn}
\author{Weiyi Zhang}
\address{Mathematics Institute\\  University of Warwick\\ Coventry, CV4 7AL, England}
\email{weiyi.zhang@warwick.ac.uk}

\begin{abstract}
This is the first of a series of papers, in which we study the plurigenera, the Kodaira dimension and more generally the Iitaka dimension on compact almost complex manifolds.

Based on the Hodge theory on almost complex manifolds, we introduce the plurigenera, Kodaira dimension and Iitaka dimension on compact almost complex manifolds. We show that the plurigenera and the Kodaira dimension as well as the irregularity are birational invariants in almost complex category, at least in dimension $4$, where a birational morphism is defined to be a degree one pseudoholomorphic map. However, they are no longer deformation invariants, even in dimension $4$ or under tameness assumption. On the way to establish the birational invariance, we prove the Hartogs extension theorem in the almost complex setting by the foliation-by-disks technique.

Some interesting phenomena of these invariants are shown through examples. In particular, we construct non-integrable compact almost complex manifolds with large Kodaira dimensions. Hodge numbers and plurigenera are computed for the standard almost complex structure on the six sphere $S^6$, which are different from the data of a hypothetical complex structure.

\end{abstract}

\date{}
\maketitle

\tableofcontents

\section{Introduction}
The Iitaka dimension for a holomorphic line bundle $L$ over a compact complex manifold is a numerical invariant to measure the size of the space of holomorphic sections. It could be equivalently defined as the growth rate of the dimension of the space $H^0(X, L^{\otimes d})$, or the maximal image dimension of the rational map to projective space determined by powers of $L$, or $1$ less than the dimension of the section ring of $L$. The Iitaka dimension of the canonical bundle $\mathcal K_X$ of a compact complex manifold $X$ is called its Kodaira dimension and $H^0(X, \mathcal K_X^{\otimes d})$ is called the $d$-th plurigenus.

The Kodaira dimension, plurigenera and the canonical section ring are birational invariants. They play important roles in the study of complex manifolds. In particular, the Kodaira dimension is used to give a rough birational classification of complex manifolds. It is known that Kodaira dimension is a deformation invariant for compact complex surfaces, although it is no longer true for complex (non-K\"ahler) $3$-folds \cite{Nak}. Siu \cite{S1},\cite{S2} shows that plurigenera, and thus also the Kodaira dimension, are invariant with respect to projective deformations of algebraic varieties. Birkar-Cascini-Hacon-McKernan \cite{BCHM} shows that the canonical ring of a smooth projective variety is finitely generated, which implies that there is a unique canonical model for every variety of general type.

The theory of complex manifolds lies in the more general framework of almost complex manifolds. In \cite{Z2}, intersection theory of almost complex manifolds is introduced. As a consequence, pseudoholomorphic degree one maps are considered as birational morphisms in the almost complex category. An important step to develop birational geometry for almost complex manifolds is to introduce and study birational invariants.

In this series of papers, we will generalize the notions of Kodaira dimension, plurigenera as well as the space of holomorphic $p$-forms to compact almost complex manifolds. The crucial initial step is to have generalizations of holomorphic line bundle and its holomorphic sections. We have two equivalent versions.

The first is from differential geometry. A {\it pseudoholomorphic structure} on a complex vector bundle $E$ over an almost complex manifold $X$ is a differential operator $\bar{\p}_E$ acting on smooth sections which satisfies the  Leibniz rule (Definition \ref{ps}). In particular, the canonical bundle has a natural pseudoholomorphic structure inherited from the almost complex structure on $X$.  By Koszul-Malgrange theorem, $\bar{\p}_E^2=0$ on a complex manifold is equivalent to a holomorphic structure on the complex bundle $E$. Our generalized version of holomorphic sections is just the smooth sections in the kernel of $\bar{\p}_E$.

To show these generalized holomorphic sections are of finite dimenstion, we apply the method of Hodge theory. Hodge theory is well developed on compact complex manifolds and on general compact Riemannian manifolds. On the way of making sense of the counting of pseudoholomorphic sections and defining plurigenera for almost complex manifolds, we develop the Hodge theory for Hermitian bundles over compact almost complex manifolds in details and show the following theorem.

\begin{thm}
Let $E$ be a complex vector bundle with a pseudoholomorphic structure over a compact almost complex manifold $(X, J)$. Then $H^0(X, E)$ is finite dimensional. In particular, $H^0(X, \mathcal K^{\otimes m})$ is finite dimensional and an invariant of $J$.
\end{thm}
This result gives us a good base to count pseudoholomorphic sections. In fact, we are able to define Dolbeault harmonic forms $\mathcal{H}_{\bar{\p}_{E}}^{(p,q)}(X, E)$ which give Dolbeault type cohomology groups when $q=0$. The vector space $H^0(X, E)$ is simply the case of $p=q=0$. When $E$ is the trivial bundle, the space of harmonic forms of type $(p,q)$ has been defined in \cite{Hir}. The Problem 20 in Hirzebruch's list\cite{Hir}, raised by Kodaira and Spencer, asks whether their dimensions are independent of the choice of the Hermitian structure.\footnote{Recently, this problem is answered negatively in \cite{HZ}.} Our discussion gives affirmative answer to this problem when $q=0$ or, by Serre duality (Proposition \ref{Serre}), $q=\dim_{\C} X$. 

The second description of pseudoholomorphic sections is more geometric. There are special almost complex structures on the total space of the complex vector bundle $E$, called {\it bundle almost complex structures}, introduced by De Bartolomeis-Tian in \cite{DeT}. The authors show that there is a bijection between bundle almost complex structures and the pseudoholomorphic structures on $E$ (also see Proposition \ref{bacs=ps}). We further observe, in Corollary \ref{holo}, that a section in the kernel of a pseudoholomorphic structure $\bar{\p}_E$ is exactly a pseudoholomorphic section with respect to the bundle almost complex structure $\mathcal J$ corresponding to $\bar{\p}_E$.

With these two equivalent descriptions understood, we are able to give our definition of $(E, \mathcal J)$-genus and the definition of Iitaka dimension, as well as their special cases - the plurigenera and  the Kodaira dimension in the end of Section \ref{defIita}.

\begin{defn}\label{Iv1}
Let $E$ be a complex vector bundle with bundle almost complex structure $\mathcal J$ over an almost complex manifold $(X,J)$. The $(E, \mathcal J)$-genus is defined as $$P_{E, \mathcal J}:=\dim H^0(X, (E, \mathcal J)),$$ where $H^0(X, (E, \mathcal J))$ denotes the space of $(J,\mathcal J)$ pseudoholomorphic sections. The $m^{th}$ plurigenus of $(X,J)$ is defined to be $P_m(X, J)=\dim H^0(X, \mathcal K^{\otimes m})$.

Let $L$ be a complex line bundle with bundle almost complex structure $\mathcal J$ over $(X, J)$. The Iitaka dimension $\kappa^J(X, (L, \mathcal J))$ is defined as
$$\kappa^J(X, (L, \mathcal J))=\begin{cases}\begin{array}{cl} -\infty, &\ \text{if} \ P_{L^{\otimes m},\mathcal J}=0\  \text{for any} \ m\geq 0\\
\limsup_{m\rightarrow \infty} \dfrac{\log P_{L^{\otimes m},\mathcal J}}{\log m}, &\  \text{otherwise.}
\end{array}\end{cases}$$

The Kodaira dimension $\kappa^J(X)$ is defined by choosing $L=\mathcal K$ and $\mathcal J$ to be the bundle almost complex structure induced by $\bar{\p}$.
\end{defn}

The advantage to have the second description is that the intersection theory of almost complex submanifolds developed by the second author in \cite{Z2} can come into play. The theory works particularly well when the base manifold $(X, J)$ is of dimension $4$. In this situation, the zero locus of a pseudoholomorphic section is a $J$-holomorphic curve in the first Chern class of the complex bundle $E$. With this understood, the rich theory of pseudoholomorphic curves are in our armory.

As we mentioned above, the plurigenera and thus the Kodaira dimension are classical birational invariants. We certainly expect it is true for our $P_m(X)$ and $\kappa^J(X)$.
As suggested by Theorem 1.5 in \cite{Z2}, a degree one pseudoholomorphic map is the right notion of birational morphism in the almost complex setting, at least in dimension $4$. The next result, as a combination of Theorems \ref{Kodbir} and \ref{hodgebir}, confirms the birational invariance of plurigenera, Kodaira dimension and the irregularity $h^{1,0}(X):=\dim H^{0}(X, \Omega^p(\mathcal O))$.

\begin{thm}\label{birintro}
Let $u: (X, J_X)\rightarrow (Y, J_Y)$ be a degree one pseudoholomorphic map between closed almost complex $4$-manifolds. Then $P_m(X, J_X)=P_m(Y, J_Y)$ and thus $\kappa^{J_X}(X)= \kappa^{J_Y}(Y)$. Moreover, the Hodge number $h^{1,0}(X)=h^{1,0}(Y)$.
\end{thm}
The most essential ingredient is to establish the desired  Hartogs extension theorem in the almost complex setting, which certainly has its independent interest. It is only established in dimension $4$ by the foliation-by-disks technique (see {\it e.g.} \cite{T96, Z2}). It is Theorem \ref{hartogs} which we reproduce in the following.

\begin{thm}\label{hartogsintro}
Let $(E, \mathcal J)$ be a complex vector bundle with a bundle almost complex structure over the almost complex $4$-manifold $(X, J)$, and $p\in X$. Then any section in $H^0(X\setminus p, (E, \mathcal J)|_{X\setminus p})$ extends to a section in $H^0(X, (E, \mathcal J))$.
\end{thm}

The next step is to study the property of plurigenera under deformation of almost complex structures. For projective manifolds, the plurigenera are invariant under projective deformation. On complex surfaces, the plurigenera (hence the Kodaira dimension) are even diffeomorphism invariants \cite{FM, FQ}, although it is no longer true when the complex dimension is greater than $2$ (see \cite{R}). Moreover, the irregularity of a complex surface is a homotopy invariant.

By virtue of our Hodge theoretic description of plurigenera, they are upper semi-continuous functions under smooth deformation. However, it is easy to see that the dimensions could jump. When we deform an integrable almost complex structure of a surface of general type, a generic perturbed almost complex structure does not admit any pseudoholomorphic curve, while as mentioned above the zero locus of a non-trivial pseudoholomorphic section of a pluricanonical bundle is a pseudoholomorphic curve in the class $mK$. This argument itself does not exclude the possibility of invariance when the canonical class is torsion. In Section 6, we construct some explicit deformations on Kodaira-Thurston surface and $4$-torus, and show that the plurigenera, the Kodaira dimension and the irregularity are not constant under smooth deformation even when the canonical class is trivial.

Also in Section 6, we study the relation between non-integrability of almost complex structures and the Kodaira dimension. Namely, we search the possible values of Kodaira dimension if the almost complex structure is non-integrable. By applying the Riemann-Roch formula and the almost complex K\"unneth formula, we prove the following result (Theorem \ref{niKod}).

\begin{thm}
For every $k\in\{-\infty, 0, 1, \cdots, n-1\}$, $n\geq 2$, there are examples of compact 2n-dimensional non-integrable almost complex
manifolds with $\kappa^J=k$.
\end{thm}

We point out that the above range of Kodaira dimension for non-integrable almost complex structure is optimal. More precisely, we will show that if the Kodaira dimension equals the complex dimension of the manifold, then the almost complex structure must be integrable. This will appear in the second paper.

In the last section of the paper, we compute the Hodge numbers, the plurigenera and the Kodaira dimension on the six sphere $S^6$. It is classically known that there exist almost complex structures on $S^6$ \cite{Eh}. A standard construction is to use the cross product of $\mathbb R^7$ applying to the tangent space of $S^6$. Denote this standard almost complex structure by $\mathsf J$. In Theorem \ref{sphere}, we prove that

\begin{thm}
For the standard almost complex structure $\mathsf J$ on $S^6$, the following hold: (1) $h^{1,0}=h^{2,0}=h^{2,3}=h^{1,3}=0$; (2) $P_m(S^6, \mathsf J)=1$ for any $m\geq 1$ and $\kappa^{\mathsf J}=0$.
\end{thm}

This calculation is somewhat surprising since it is generally believed that the Kodaira dimension of a hypothetical complex structure is $-\infty$. Our plurigenera distinguish $\mathsf J$ from hypothetical complex structures on $S^6$, since for the latter $P_1=h^{3, 0}=0$.

In paper II, we will interpret the Kodaira dimension through the pluricanonical map and discuss the significant geometric consequences. We will also investigate its comparison with the symplectic Kodaira dimension \cite{L} on symplectic 4-manifolds. Some vanishing theorems on positively-curved almost Hermitian manifolds will also be proved.\\

\textbf{Acknowledgements} The authors are kindly informed by Tian-Jun Li that he has a joint project with Gabriel La Nave on Kodaira dimension for almost K\"ahler manifolds with a totally different strategy. The first author would also like to thank Professors Bo Guan, Jiaping Wang and Fangyang Zheng for their encouragement and thank Xiaolan Nie for her support.

\section{Notations}
We start by fixing our notations and explain the natural pseudoholomorphic structure on the pluricanonical bundles.

Let $(X,J)$ be a $2n$-dimensional almost complex manifold. The complexification of the cotangent bundle of $X$ decomposes as $T^*X\otimes \mathbb C=(T^*X)^{1, 0}\oplus (T^*X)^{0,1}$ where $(T^*X)^{1,0}$ annihilates the subspace in $TX\otimes \mathbb C$ where $J$ acts as $-i$. A $(1,0)$-form is a smooth section of $(T^*X)^{1, 0}$; similarly for a $(0, 1)$-form. The splitting of the cotangent bundle induces a splitting of all exterior powers. Write $\Lambda^{p, q}X=\Lambda^p((T^*X)^{1, 0})\otimes \Lambda^q((T^*X)^{0, 1})$. Then for any $r\ge 0$, we have the decomposition $$\Lambda^rT^*X\otimes \mathbb C=\oplus_{p+q=r}\Lambda^{p, q}X.$$ Let $\pi^{p,q}$ be the projection to $\Lambda^{p, q}X$. A $(p, q)$-form is a smooth section of the bundle $\Lambda^{p, q}X$. The space of all such sections is denoted $\Omega^{p, q}(X)=\Gamma(X, \Lambda^{p,q})$.

The $\bar{\partial}$ and $\partial$ operator can be defined by:
$$\bar{\partial}=\pi^{p,q+1}\circ d: \Omega^{p,q}(X)\rightarrow \Omega^{p,q+1}(X)$$
$$\partial=\pi^{p+1,q}\circ d: \Omega^{p,q}(X)\rightarrow \Omega^{p+1,q}(X),$$
where $d$ is the exterior differential. Both $\bar{\partial}$ and $\partial$ satisfy the Leibniz rule, but in general $\bar{\partial}^2$ and $\partial^2$ may not be zero. They contain important information of almost complex structures. Apply $\bar{\p}$ to $\Lambda^{p, 0}$ and in particular $\mathcal K=\Lambda^{n,0}$, we have$$\bar{\partial}: \Lambda^{p,0}\rightarrow \Lambda^{p,1}\cong (T^*X)^{0,1} \otimes \Lambda^{p,0},$$
 $$\bar{\partial}: \mathcal K\rightarrow \Lambda^{n,1}\cong (T^*X)^{0,1} \otimes \mathcal K.$$ Here we write $\mathcal K$ (or any vector bundle) in short for any smooth sections of $\mathcal K$ (the vector bundle). We can extend the $\bar{\partial}$ to an operator $\bar{\partial}_m: \mathcal K^{\otimes m}\rightarrow (T^*X)^{0,1}\otimes \mathcal K^{\otimes m}$ for $ m\geq 2$ inductively by the product rule
$$\bar{\partial}_m(s_1\otimes s_2)=\bar{\partial}s_1\otimes s_2+s_1\otimes \bar{\partial}_{m-1} s_2.$$
It satisfies the Leibniz rule $\bar{\p}_m (fs)=\bar{\p}f\otimes s+f\bar{\p}_m s$ for any section $s$ of $\mathcal K^{\otimes m}$ and $\Lambda^{p,0}$.
\begin{defn}
The space of holomorphic sections of $\mathcal K^{\otimes m}$ is defined to be
$$H^0(X, \mathcal K^{\otimes m})=\{s\in \Gamma(X, \mathcal K^{\otimes m}): \bar{\partial}_m s=0\}.$$
\end{defn}
\begin{remk} The space $H^0(X, \mathcal K^{\otimes m})$ is categorical in the almost complex category. Indeed, if $(X', J')$ is another almost complex manifold which has a diffeomorphism $F: X'\rightarrow X$ to $(X,J)$ satisfying $dF\circ J'=J\circ dF$, we say that $(X', J')$ is \textit{pseudoholomorphic isomorphic} to $(X,J)$. Then $F^*(\Omega^{p,q}(X))=\Omega^{p,q}(X')$ and $F^*\circ \bar{\p}_J=\bar{\p}_{J'}\circ F^*$. So $s\in \Gamma(X, \mathcal K_J)$ satisfying $\bar{\p}s=0$ if and only if $\bar{\p}_{J'} F^* s=0$. Similar result holds on $\mathcal K_J^{\otimes m}$ where $F^*$ and $\bar{\p}$ are replaced by an isomorphism $F^*_m$ and the operator $\bar{\p}_m$. Therefore, $F$ induces an isomorphism $F^*_m: H^0(X, \mathcal K_J^{\otimes m})\rightarrow H^0(X',\mathcal K_{J'}^{\otimes m})$ for any $m\geq 1$.
\end{remk}

\section{Hodge theory on almost complex manifolds}

In this section, we will define Dolbeault cohomology groups for a complex bundle associated with a pseudoholomorphic structure. We show that they are finite dimensional when $X$ is compact in Theorem \ref{finite}. As a consequence, $H^0(X,\mathcal K^{\otimes m})$ is finite dimensional. We will follow the method of Hodge theory to define a formal adjoint operator of $\bar{\partial}_m$ and apply the elliptic theory.

Hodge theory is well developed on compact complex manifolds (see \cite{GH}, \cite{Huy}, \cite{V}, \cite{Z}). The Hodge theory on compact almost complex manifolds was initiated in \cite{Hir} for $(p,q)$ forms. To derive our results, we will set up the Hodge theory for $(p,q)$ forms with value in a Hermitian pseudoholomorphic bundle $E$. We then apply it to complex line bundles, in particular the bundle $\mathcal K^{\otimes m}$ and show that $H^0(X, \mathcal K^{\otimes m})$ is finite dimensional. One of our observations is that for a holomorphic vector bundle $E$ over a complex manifold, the $\bar{\p}$ operator on the holomorphic dual $E^*$ coincides with the $(0,1)$ part of a Hermitian connection on $\bar{E}$ when identifying $E^*$ with $\bar{E}$ by a Hermitian metric.

Choose a Riemannian metric $g$ on $X$ compatible with $J$, namely $g(Ju, Jv)=g(u,v)$ for any $v, w\in TX$. Then $g$ induces a Hermitian structure $h$ on $TX\otimes \mathbb{C}$ by $h=g-i\omega$, where $\omega(u, v)=g(Ju, v)$. Also, $h$ can be extended to Hermitian structures on the bundles $\Lambda^{p,q}X$ for any $(p, q)$ which we still denote by $h$. Explicitly, assume that $\{e_i\}$ is a local unitary frame in $TX\otimes \mathbb{C}$ and $\{\phi_i\}$ is the unitary coframe so that $$h=\sum^n_{i=1} \phi_i\otimes \bar{\phi}_i.$$ If $\alpha=\{i_1,i_2,\cdots, i_p\}, \beta=\{j_1,j_2,\cdots, j_q\}$ is any ordered set of $(p,q)$ multiindices, denote $\phi_{\alpha}=\phi_{i_1}\wedge\phi_{i_2}\wedge\cdots\wedge \phi_{i_p}$, $\bar{\phi}_{\beta}=\bar{\phi}_{j_1}\wedge\bar{\phi}_{j_2}\wedge\cdots\wedge \bar{\phi}_{j_p}$. Then $h$ on  $\Lambda^{p,q}$ is defined by letting $\{\phi_{\alpha}\wedge \bar{\phi}_{\beta}\}$ be orthogonal and $h(\phi_{\alpha}\wedge \bar{\phi}_{\beta}, \phi_{\alpha}\wedge \bar{\phi}_{\beta})=2^{p+q}$.

The $*$ operator on an almost Hermitian manifold is the unique $\mathbb{C}$-linear operator $$*: \Lambda^{p,q}\rightarrow \Lambda^{n-q,n-p}$$ satisfying \begin{align} h(\varphi_1, \varphi_2)dV=\varphi_1\wedge\overline{*\varphi_2}\end{align} where $dV$ is the volume form of $g$ and $\varphi_1, \varphi_2\in \Lambda^{p,q}.$

Using the unitary coframe $\{\phi_i\}$, we can write out the $*$ operator directly. Let $\hat{\alpha}, \hat{\beta}$ be the ordered set of the complement multi-indices of $\alpha, \beta$ in $\{1,2,\cdots,n\}$. As $\omega=i\sum_{i=1}^n \phi_i\wedge \bar{\phi}_i$ and $dV=\dfrac{\omega^n}{n!}$, if we define on the basis \begin{align} *(\phi_{\alpha}\wedge\bar{\phi}_{\beta})=2^{(p+q-n)}(-i)^n\epsilon_{\alpha\beta\hat{\beta}\hat{\alpha}} \phi_{\hat{\beta}}\wedge\bar{\phi}_{\hat{\alpha}} \end{align} where $\epsilon_{\alpha\beta\hat{\beta}\hat{\alpha}}$ is the sign of permutation of $$(i_1,\cdots, i_p, j_1,\cdots, j_q, \hat{j}_1,\cdots,\hat{j}_{n-q},\hat{i}_1,\cdots,\hat{i}_{n-p})\rightarrow (1,1',2,2'\cdots,n,n'),$$
then the operator of (2) satisfies (1). By uniqueness it gives the $*$ operator.

Define an inner product on $\Omega^{p,q}(X)$ by $\langle \varphi_1, \varphi_2\rangle=\int_X h(\varphi_1,\varphi_2)dV$ for $\varphi_1,\varphi_2 \in \Omega^{p,q}(X)$. Let $$\bar{\partial}^*=-*\partial*.$$ Then the following holds $$\langle \bar{\partial}\varphi_1, \varphi_2\rangle=\langle \varphi_1, \bar{\partial}^*\varphi_2\rangle.$$ The proof is the same with the integrable case, because the Leibniz rule holds and $\bar{\partial}=d$ acting on $A^{(n,n-1)}$.
Indeed, by (1), we have \begin{align*} \langle \bar{\partial}\varphi_1, \varphi_2\rangle&= \int_X h(\bar{\partial}\varphi_1,\varphi_2)dV=\int_X \bar{\partial}\varphi_1\wedge \overline{*\varphi_2}\\
&=\int_X \bar{\partial}(\varphi_1\wedge \overline{*\varphi_2})-(-1)^{p+q-1}\int_X \varphi_1\wedge \overline{\partial*\varphi_2}\\
&=\int_X \varphi_1\wedge \overline{*(\bar{\partial}^*\varphi_2)}\\
&=\int_X h(\varphi_1,\bar{\partial}^*\varphi_2)dV=\langle \varphi_1, \bar{\partial}^*\varphi_2\rangle,
\end{align*}
where we use the Stokes' theorem in the third line.

The above discussion produces the formal dual operator of $\bar{\p}$ on $\Omega^{p,q}(X)$. The next important step is to generalize this operator to any Hermitian bundle $E$ with a pseudoholomorphic structure.

\begin{defn} Let $(E,h_E)$ be a Hermitian vector bundle over $(X,J)$. A connection $\nabla: \Gamma(X,E)\rightarrow \Gamma(X, (T^*X\otimes \C)\otimes E)$ is called a Hermitian connection if \begin{align} \label{def} d (h_E(s_1, s_2))=h_E(\nabla s_1, s_2)+h_E(s_1, \nabla s_2),\end{align}
for any two sections $s_1, s_2$ of $E$. \end{defn}

\begin{defn}\label{ps}
A pseudoholomorphic structure on $E$ is given by a differential operator $\bar{\partial}_E: \Gamma(X, E)\rightarrow  \Gamma(X, (T^*X)^{0,1}\otimes E)$ which satisfies the Leibniz rule $$\bar{\partial}_E (fs)=\bar{\partial}f\otimes s+f\bar{\partial}_Es$$ where $f$ is a smooth function and $s$ is a section of $E$.
\end{defn}
If the pseudoholomorphic structure $\bar{\p}_E$ satisfying $\bar{\partial}_E^2=0$ on a complex manifold, it is equivalent to a holomorphic structure on the complex bundle $E$ by Koszul-Malgrange's theorem. In particular, any pseudoholomorphic structure on a complex vector bundle over a Riemann surface $S$ is holomorphic, since $(T^*S)^{0,2}=0$.

Denote $\nabla^{(1,0)}, \nabla^{(0,1)}$ the $(1,0)$ and $(0,1)$ components of $\nabla$. We have
\begin{lem}\label{cancon}
For any Hermitian bundle $(E, h_E)$ with a pseudoholomorphic structure $\bar{\p}_E$, there is a unique Hermitian connection $\nabla$ so that $\nabla^{(0,1)}=\bar{\partial}_E$.

\end{lem}

The lemma is well known when $J$ is integrable and should be known to experts for general $J$ (see \cite{DeT}). We include a proof for convenience of readers.

\begin{proof} We first prove the existence. Assume that $\{U_\a\}$ is an open chart covering of $X$ with partition of unity $\{\varphi_\a\}$ such that $E|_{U_\a}$ is trivial. For any $s\in \Gamma(X,E)$ and any connection $\nabla$, $\nabla s=\nabla\sum_{\a} \varphi_\a s=\sum_{\a} \nabla(\varphi_\a s)$. So we only need to define $\nabla$ locally on $U_\a$.

Let $\{s_{i}, 1\leq i\leq N\}$ be a unitary frame of $E$ on $U_\a$. Using summation notation, denote $\bar{\p}_Es_i=\theta_i^js_j$, where $\theta_i^j\in T^{0,1}$. Let $\{s'_{i}, 1\leq i\leq N\}$ be another unitary frames of $E$ with $s'_i=f_i^j s_j$. As $\sum_jf_i^j\bar{f}_k^j=\delta_{ik}$, $(f^{-1})_i^j=\bar{f}^i_j$. Denote $\bar{\p}_Es'_i=(\theta')_i^js_j'$. We have
$$(\theta')_i^j=\sum_k(\bar{\p}f_i^k+f_i^l\theta_l^k)\bar{f}^k_j.$$
To define $\nabla$, let $\omega^j_i=\theta_i^j-\overline{\theta_j^i}$ and $\nabla s_{i}=\omega^j_i s_j$. Similarly, for $\{s'_{i}\}$, let $(\omega')^j_i=(\theta')_i^j-\overline{(\theta')_j^i}$ and $\nabla s'_{i}=(\omega')^j_i s'_j$. If $\{\omega^j_i\}$ and $\{(\omega')^j_i\}$ satisfy the transition equation
$(\omega')_i^j=\sum_k(df^k_i+f_i^l\omega_l^k)\bar{f}_j^k,$ they give a well defined connection. This follows by
\begin{align*}
(\omega')_i^j&=(\theta')_i^j-\overline{(\theta')_j^i}\\
&=\sum_k((\bar{\p}f_i^k+f_i^l\theta_l^k)\bar{f}^k_j-(\p\bar{f}_j^k+\bar{f}_j^l\overline{\theta_l^k})f^k_i)\\
&=\sum_k(df^k_i+f_i^l(\theta_l^k-\overline{\theta_k^l}))\bar{f}^k_j=\sum_k(df^k_i+f_i^l\omega_l^k)\bar{f}_j^k
\end{align*}
 where we use $\p(\sum_kf_i^k\bar{f}_j^k)=0$ for the third equality. So $\nabla$ is independent of the frames. From the skew symmetry of $\nabla$, we know that it is a Hermitian connection compatible with $h_E$.

The uniqueness follows easily if we restrict $\nabla$ to the open chart above.
\end{proof}

\begin{remk}\label{chern}
Recall that the almost Chern connection \cite{EL}(see also \cite{Gau2}) associated to $g$ is the unique connection $\nabla^c$ on the tangent bundle such that $\nabla^c J=\nabla^c g=0$ and that the torsion $\Theta$ has vanishing $(1,1)$ part. The $\bar{\p}$ operator on $(T^*X)^{1,0}$ induces a natural pseudoholomrphic structure.  It turns out that the unique Hermitian connection on $(T^*X)^{1,0}$ induced by $\bar{\p}$ as in Lemma \ref{cancon} equals $\nabla^c$. To see this, assume that $\nabla^c e_i=\omega_i^j e_j$ for a unitary frame $\{e_i\}$. By the first structure equation, the $i^{th}$ component of $\Theta$ is $\Theta^i=d \phi_i+\omega_j^i \wedge \phi_j$, where $\{\phi_i\}$ is the coframe. Also $\nabla^c$ acts on $(T^*X)^{1,0}$ by $\nabla^c \phi_i=-\omega^i_j\phi_j$. Then $\Theta^i$ has vanishing $(1,1)$ part if and only if $\bar{\p}\phi_i+(\omega_j^i)^{0,1} \wedge \phi_j=0$ which is equivalent to $(\nabla^c)^{(0,1)}=\bar{\p}$.

Now suppose $E$ is the pluricanonical bundle $\mathcal K_J^{\otimes m}$ with the induced pseudoholomorphic structure $\bar{\p}_m$. Following from the above discussion, the unique Hermitian connection on $\mathcal K_J^{\otimes m}$ induced by the Chern connection on $(T^*X)^{1,0}$ is just the unique Hermitian connection determined by $\bar{\p}_m$ in Lemma \ref{cancon}.

\end{remk}

Let $(E,h_E)$ be the Hermitian bundle with a pseudoholomorphic structure $\bar{\p}_E$. We can define a unique dual pseudoholomorphic structure on $E^*$: $$\bar{\p}_{E^*}: \Gamma(X, E^*)\rightarrow \Gamma(X, (T^*X)^{0,1}\otimes E^*)$$ as follows. For any section $s^*\in \Gamma(X, E^*)$ and any section $s'\in \Gamma(X, E)$, let
\begin{align}\label{L} (\bar{\p}_{E^*}(s^*))(s')=\bar{\p} (s^*(s'))-s^*(\bar{\p}_E(s')). \end{align}
It is easy to verify that $\bar{\p}_{E^*}$ satisfies the Leibniz rule, giving a pseudoholomorphic structure.
With the Hermitian structure $h_E$, there exists a natural complex linear isomorphism $E^*\cong \bar{E}$, where $\bar{E}$ is the conjugate bundle of $E$. Therefore, $\bar{\p}_{E^*}$ induces a pseudoholomorphic structure on $\bar{E}$. On the other side, by Lemma \ref{cancon}, there is a unique Hermitian connection $\nabla$ on $E$ determined by $\bar{\p}_E$ and $h_E$. The conjugate of the $(1,0)$ part of $\nabla$ induces $$\overline{\nabla^{(1,0)}}: \bar{E}\rightarrow (T^*X)^{0,1}\otimes \bar{E}.$$
Define $\bar{\p}_{\bar{E}}=\overline{\nabla^{(1,0)}}$. We have

 \begin{lem}\label{lem3.5}

By identifying $\bar{E}$ with $E^*$, $\bar{\p}_{\bar{E}}=\bar{\p}_{E^*}.$

 \end{lem}

 \begin{proof} Let $\bar{s}\in \bar{E}$ and $s'\in E$. The inner product $h_E$ on $E$ induces a bilinear paring between $E$ and $\bar{E}$ which we still denote by $h_E$. Then by (\ref{def}), \begin{align} \label{bar} \bar{\p}h_E(s',\bar{s})=h_E(\nabla^{0,1}s',\bar{s})+h_E(s',\overline{\nabla^{(1,0)}}\bar{s})=h_E(\bar{\p}_Es',\bar{s})+h_E(s',\bar{\p}_{\bar{E}}\bar{s}).\end{align} Therefore, $\bar{\p}_{\bar{E}}$ satisfies the product rule (\ref{L}). Therefore, $\bar{\p}_{\bar{E}}=\bar{\p}_{E^*}$.
 \end{proof}

Next we can extend the $\bar{\p}_E$ operator to $\Lambda^{p,q}\otimes E$ and $\bar{\p}_{\bar{E}}$ to $\Lambda^{r,s}\otimes \bar{E}$ by \begin{align} \label{tensor} \bar{\p}_E (\varphi\otimes u)=(\bar{\p}\varphi)\otimes u+(-1)^{p+q}\varphi\wedge \bar{\p}_E u\\ \notag \bar{\p}_{\bar{E}} (\phi\otimes v)=(\bar{\p}\phi)\otimes v+(-1)^{r+s}\phi\wedge \bar{\p}_{\bar{E}} v.\end{align} Then there is a wedge pairing  $$\wedge: (\Lambda^{p,q}\otimes E)\ \times\ (\Lambda^{r,s}\otimes
\bar{E})\rightarrow  \Lambda^{p+r,q+s}$$ defined by $(\varphi_1\otimes u)\wedge (\varphi_2\otimes v)=h_E(u,v)\varphi_1\wedge\varphi_2$. As before, in the situation, $h_E(u,v)$ is denoted to be the $\mathbb{C}$-bilinear product between $E$ and $\bar{E}$. We have the Leibniz rule for the wedge pairing
\begin{align}\label{Leibniz} \bar{\p}_E(\varphi_1\otimes u)\wedge (\varphi_2\otimes v)=&(\bar{\p}\varphi_1\otimes u+(-1)^{p+q}\varphi_1\wedge \bar{\p}_E u)\wedge (\varphi_2\otimes v)  \notag \\
=&h_E(u,v)\bar{\p}\varphi_1\wedge\varphi_2+(-1)^{2(p+q)}h_E(\bar{\p}_E u,v)\wedge\varphi_1\wedge\varphi_2  \notag\\
=&h_E(u,v)(\bar{\p}(\varphi_1\wedge\varphi_2)-(-1)^{p+q}\varphi_1\wedge\bar{\p}\varphi_2)\\
& +(\bar{\p}h_E(u,v)-h_E(u, \bar{\p}_{\bar{E}} v))\wedge\varphi_1\wedge\varphi_2 \notag \\
=&\bar{\p}((\varphi_1\otimes u)\wedge (\varphi_2\otimes v))-(-1)^{p+q}(\varphi_1\otimes u)\wedge \bar{\p}_{\bar{E}}(\varphi_2\otimes v) \notag
\end{align}
where we use (\ref{bar}) in the third line and $h(u, \bar{\p}_{\bar{E}} v)\wedge\varphi_1\wedge\varphi_2 = (-1)^{p+q+r+s}(\varphi_1\otimes u)\wedge (\varphi_2\otimes \bar{\p}_{\bar{E}}v)$ for the fourth line.

Now we are able to find the dual operator $\bar{\p}^*_E$.  Define $$*: \Lambda^{p,q}\otimes E\rightarrow \Lambda^{n-q,n-p}\otimes E$$ by $*(\varphi\otimes u)=(*\varphi)\otimes u$ for any $\varphi\in \Lambda^{p,q}, u\in E$. From the definition we have
$$h_E(\varphi_1\otimes u_1, \varphi_2\otimes u_2)dV=\varphi_1\otimes u_1\wedge \overline{*(\varphi_2\otimes u_2)}$$
where $\varphi_1\otimes u_1, \varphi_2\otimes u_2\in \Lambda^{p,q}\otimes E$, $h_E$ denotes the original inner product on $E$ and $dV$ is the volume form of $X$. The inner product on $\Gamma(X, \Lambda^{p,q}\otimes E)$ is given by $$
 \langle \varphi_1\otimes u_1, \varphi_2\otimes u_2\rangle =\int_X h_E(\varphi_1\otimes u_1, \varphi_2\otimes u_2)dV.$$
 Define $$\bar{\p}_E^*=-*\nabla^{(1,0)}*.$$ We have
\begin{align*}
\langle \bar{\p}_E(\phi\otimes w), \varphi\otimes u\rangle&=\int_X h_E(\bar{\p}_E(\phi\otimes w), \varphi\otimes u)dV\\
&=\int_X\bar{\p}_E(\phi\otimes w)\wedge \overline{*(\varphi\otimes u)}\\
&=\int_X \bar{\p}(\phi\otimes w\wedge \overline{*(\varphi\otimes u)})-(-1)^{p+q-1}(\phi\otimes w)\wedge \bar{\p}_{\bar{E}}(\overline{*(\varphi\otimes u)})\\
&=-\int_X (\phi\otimes w)\wedge (-1)^{p+q-1}\overline{\nabla^{(1,0)}*(\varphi\otimes u)}\\
&=\langle \phi\otimes w, \bar{\p}_E^*(\varphi\otimes u) \rangle
\end{align*}
for $\phi\otimes w\in \Lambda^{p,q-1}\otimes E, \varphi\otimes u\in \Lambda^{p,q}\otimes E$. We use \eqref{Leibniz} in the third line and Stokes' theorem in the fourth line. So $\bar{\p}_E^*$ gives the formal adjoint of $\bar{\p}_E$.

Define the Laplacian \begin{align}\label{lap} \Delta_{\bar{\p}_E}=\bar{\p}_E\bar{\p}_E^*+\bar{\p}_E^*\bar{\p}_E.
\end{align}
As $\langle \Delta_{\bar{\p}_E} s, s\rangle=\langle \bar{\p}_E s, \bar{\p}_E s\rangle +\langle \bar{\p}_E^*s, \bar{\p}_E^* s\rangle$, $\Delta_{\bar{\p}_E} s=0$ if and only if $\bar{\p}_E s=0$ and $\bar{\p}_E^*s=0$. Denote the space of harmonic $(p,q)$ form section of $E$ by:
$$\mathcal{H}_{\bar{\p}_E}^{(p,q)}(X,E)=\{s\in \Gamma(X, \Lambda^{p,q}\otimes E)| \Delta_{\bar{\p}_E} s=0\}.$$
\begin{thm} \label{finite}
$\Delta_{\bar{\p}_E}$ is an elliptic differential operator and $\mathcal{H}_{\bar{\p}_E}^{(p,q)}(X,E)$ is finite dimensional.
\end{thm}
When $E$ is the trivial bundle, this was pointed out in \cite{Hir}.

\begin{proof}
 We first prove the case when $E$ is a trivial line bundle with $\bar{\p}_E=\bar{\p}$. We shall show that $\Delta_{\bar{\p}}$ is elliptic at any point $p\in X$. As it is a local property, it suffices to discuss in a coordinate chart $U$. Let $(J_c, h_c)$ be the constant almost complex structure and Hermitian structure on $U$ with $J_c(p)=J(p), h_c(p)=h(p)$. $J_c$ is isomorphic to the canonical complex structure of open set in $\mathbb{C}^n$. Denote $\bar{\p}_c$ the "dbar" operator of $J_c$ and $*_c$ the operator corresponding to $h_c$. We have
$$(\bar{\p}\varphi-\bar{\p}_c\varphi)(p)=0, \ *(\varphi)(p)=*_c(\varphi)(p)$$
for any $\varphi\in \Gamma(U, \Lambda^{p,q})$. So $\bar{\p}$ and  $\bar{\p}_c$ differ by a differential operator whose coefficients vanish at $p$. As the principal symbol is only related to the highest degree differential, any operator from compositions of $\bar{\p}, *, \bar{\p}$ would have the same principal symbol at $p$ with the operators if we replace them by $\bar{\p}_c, *_c, \bar{\p}_c$. In particular, $\Delta_{\bar{\p}}$ has the same principal symbol with $\Delta_{\bar{\p}_c}$ at $p$. The latter is the flat Laplacian on $\mathbb{C}^n$ which is elliptic. Therefore $\Delta_{\bar{\p}}$ is elliptic at $p$ and hence everywhere.

For a complex vector bundle $E$. Let $\{\phi_i\}$ be a local unitary coframe of $(T^*X)^{1,0}$ and $\{u_{\nu}\}$ a unitary frame of $E$. Using Einstein summation notation, any section $s$ of $\Lambda^{p,q}\otimes E$ can be expressed as:
$$s=f^{\alpha\beta\nu}\phi_{\alpha}\wedge\bar{\phi}_{\beta}\otimes u_{\nu},$$
where $(\alpha, \beta)$ runs over all multi-indices of $(p,q)$. Then $$\bar{\p}_E s= \bar{\p}f^{\alpha\beta\nu}\wedge\phi_{\alpha}\wedge\bar{\phi}_{\beta}\otimes u_{\nu}+f^{\alpha\beta\nu}\bar{\p}_E(\phi_{\alpha}\wedge\bar{\phi}_{\beta}\otimes u_{\nu}).$$
The principal symbol is calculated solely from the first term.
We use $\simeq$ to mean operators with the same principal symbol. Let $\bar{\p}\otimes id$ represent the differential operator determined by $\bar{\p}\otimes id(f^{\alpha\beta\nu}\phi_{\alpha}\wedge\bar{\phi}_{\beta}\otimes u_{\nu}) := \bar{\p}(f^{\alpha\beta\nu}\phi_{\alpha}\wedge\bar{\phi}_{\beta})\otimes u_{\nu}$ and the Leibniz rule (it depends on choice of $\{u_{\nu}\}$ and is only defined locally). Let $\Delta_{\bar{\p}}\otimes id$ be the operator given by $\Delta_{\bar{\p}}\otimes id (f^{\alpha\beta\nu}\phi_{\alpha}\wedge\bar{\phi}_{\beta}\otimes u_{\nu}):= \Delta_{\bar{\p}}(f^{\alpha\beta\nu}\phi_{\alpha}\wedge\bar{\phi}_{\beta})\otimes u_{\nu}$. We have $$\bar{\p}_E\simeq \bar{\p}\otimes id \ \text{and}\  \Delta_{\bar{\p}_E}\simeq \Delta_{\bar{\p}}\otimes id.$$
We have shown above that $\Delta_{\bar{\p}}$ is elliptic. Then $\Delta_{\bar{\p}}\otimes id$ is elliptic. Therefore  $\Delta_{\bar{\p}_E}$ is elliptic.

The second part follows directly from the elliptic theory (e.g. \cite{LM}). Explicitly, there is a Green operator $G$ together with the projection operator $H: \Omega^{p,q}(X,E)\rightarrow \mathcal{H}_{\bar{\p}_E}^{(p,q)}(X,E)$ such that $\Delta_{\bar{\p}_E}\circ G+H=Id$. Also, $\Delta_{\bar{\p}_E}$ and $G$ are both Fredholm operators and $\mathcal{H}_{\bar{\p}_E}^{(p,q)}(X,E)$ is finite dimensional.
\end{proof}
Moreover, the following Serre duality also holds on compact almost complex manifolds.
\begin{prop}\label{Serre}
For any $0\leq p,q\leq n$, $\mathcal{H}_{\bar{\p}_{E}}^{(p,q)}(X, E)\cong (\mathcal{H}_{\bar{\p}_{E^*}}^{(n-p,n-q)}(X, E^*))^*$.
\end{prop}
\begin{proof} The argument is essentially the same as the classical case (see for example Proposition 4.1.15 of \cite{Huy}), except clarifying the operators on the bundle $E^*$. We only need to show that the natural paring between $\mathcal{H}_{\bar{\p}_{E}}^{(p,q)}(X, E)$ and $\mathcal{H}_{\bar{\p}_{E^*}}^{(n-p,n-q)}(X, E^*)$ is nondegenerate. For any nonzero $s\in \mathcal{H}_{\bar{\p}_{E}}^{(p,q)}(X, E)$, since $\bar{\p}_E s=\bar{\p}_E^* s=0$, by Lemma \ref{lem3.5}, we have $$\bar{\p}_{E^*}(\overline{*s})=\bar{\p}_{\bar{E}}(\overline{*s})=\overline{\nabla^{(1,0)}*s}=0,$$ and $$\bar{\p}^*_{E^*}(\overline{*s})=-*\overline{\bar{\p}_E}*(\overline{*s})=-(-1)^{(p+q)(n-p-q)}*\overline{\bar{\p}_E s}=0.$$ So $\overline{*s}\in \mathcal{H}_{\bar{\p}_{E^*}}^{(n-p,n-q)}(X, E^*)$. As $\int_X s\wedge \overline{*s}=\|s\|^2\neq 0$, the non-degeneracy stands.
\end{proof}

Since $\Lambda^{p,1}\cong (T^*X)^{0,1}\otimes \Lambda^{p,0}$, the $\bar{\p}$ operator induces a natural pseudoholomorphic structure on $\Lambda^{p,0}$ for $0\leq p\leq n$. Denote $\Omega^p(E)=\Lambda^{p,0}\otimes E$. The pseudoholomorphic structures on $\Lambda^{p,0}$ and $E$ gives a pseudoholomorphic structure on $\Omega^p(E)$. Identifying $\Lambda^{p,1}\otimes E$ with $(T^*X)^{0,1}\otimes \Omega^p(E)$ by a permutation sign, the pseudoholomorphic structure on $\Omega^p(E)$ coincides with the $\bar{\p}_E$ operator given by (\ref{tensor}). Define $$H^0(X, \Omega^p(E))=\{s\in \Gamma(X,\Omega^p(E))=\Omega^{p,0}(X, E): \bar{\p}_{E} s=0\}.$$ We have

\begin{prop}\label{finsecE}
Let $E$ be a complex vector bundle with a pseudoholomrphic structure over a compact almost complex manifold $X$, then $H^0(X, \Omega^p(E))$ is finite dimensional for $0\leq p\leq n$.
\end{prop}
\begin{proof}
As $\bar{\p}^*_{E}=0$ on  $ \Omega^{p,0}(X, E)$, $\bar{\p}_{E} s=0$ is  equivalent to $\Delta_{\bar{\p}_{E} }s=0$. So $$H^0(X, \Omega^p(E))=\mathcal{H}_{\bar{\p}_{E}}^{(p,0)}(X, E),$$ which is finite dimensional.
\end{proof}

\begin{cor}
$H^0(X,\mathcal K^{\otimes m})$ is finite dimensional.
\end{cor}

\begin{proof}
Let $E=\mathcal K^{\otimes m}$ with $\bar{\p}_E=\bar{\p}_m$ and $p=0$. Then it follows from Proposition \ref{finsecE}.
\end{proof}

Proposition \ref{finsecE} also implies that $\mathcal{H}_{\bar{\p}_{E}}^{(p,0)}(X, E)$ is independent of the Hermitian metric used to define  $\Delta_{\bar{\p}_{E}}$. This should not hold for $\mathcal{H}_{\bar{\p}_{E}}^{(p,q)}(X, E)$ for $q>0$.  However, it is possible that the dimension of $\mathcal{H}_{\bar{\p}_{E}}^{(p,q)}(X, E)$ is independent of the defining Hermitian metric. When $q=\dim_{\C}X$, Proposition \ref{Serre} implies this is true. In general, we have the following question as a generalization of Problem 20 (Kodaira-Spencer) in Hirzebruch's list \cite{Hir}.
 \begin{que}\label{Dol}
Does $\dim \mathcal{H}_{\bar{\p}_{E}}^{(p,q)}(X, E)$ depends only on $J$ and $\bar{\p}_{E}$ for any $0<q<\dim_{\C} X$?
\end{que}
\section{Bundle almost complex structure and Iitaka dimension}\label{defIita}

Recall that for any holomorphic vector bundle over a complex manifold, the total space is also a complex manifold so that any smooth section $s$ is $\bar{\p}$ closed if and only if $s$ induces a holomorphic map. For a complex vector bundle $E$ over the almost complex manifold $(X,J)$,  a {\it bundle almost complex structure} as in \cite{DeT} (here we use the rephrasement from \cite{LS}) is an almost complex structure $\mathcal J$ on $TE$ so that
\begin{enumerate}[(i)]
 \item the projection is $(\mathcal J, J)$-holomorphic,
 \item  $\mathcal J$ induces the standard complex structure on each fiber, i.e. multiplying by $i$,
 \item  the fiberwise addition $\alpha: E\times_M E\rightarrow E$ and the fiberwise multiplication by a complex number $\mu:\C\times E\rightarrow E$ are both pseudoholomorphic.
 \end{enumerate}

It is shown in \cite{DeT} that a bundle almost complex structure $\mathcal J$ on $E$ determines a pseudoholomorphic structure $\bar{\p}_{\mathcal J}$, and the map $\mathcal J\mapsto  \bar{\p}_{\mathcal J}$ is a bijection between the spaces of bundle almost complex structures and pseudoholomorphic structures on $E$.
We include here a direct proof for reader's convenience.

\begin{prop}[De Bartolomeis-Tian]\label{bacs=ps} There is a bijection between bundle almost complex structures and the pseudoholomorphic structures on $E$. \label{deT}
\end{prop}

\begin{proof}

Assume that $\mathcal{J}$ is a bundle almost complex structure. Let $s:X\lra E$ be in $\Gamma(X,E)$ and $\pi: E\lra X$ be the projection. We have $\pi\circ s=id_X$. Define $d''s: TX\lra TE$ by $d''s=\frac{1}{2}(ds+\mathcal{J}\circ ds\circ J)$. Then $d''s=0$ if and only if $s$ is a $(J,\mathcal{J})$ holomorphic map.
From $d''s$, we will define an element in $\Gamma(X,(T^*X)^{(0,1)}\otimes E)$. We have \begin{align*}
d\pi\circ d''s&=\frac{1}{2}(d\pi\circ ds+d\pi\circ \mathcal{J}\circ ds\circ J)\\
&=\frac{1}{2}(id+J\circ d\pi\circ ds\circ J)=\frac{1}{2}(id-id)=0,\end{align*}
where property i) is used in the second line. So $d''s(TX)\in ker(d\pi)=V(E)$ with $V(E)$ being the vertical tangent bundle of $E$. For a vector bundle, $V(E)$ is canonically isomorphic to $\pi^*(E)$ on $E$ and there is a commutative diagram:

\begin{center}
\begin{tikzcd}
	V(E)\cong \pi^*(E) \arrow{r}{\pi^*} \arrow{d}
& E\arrow{d}
\\
E\arrow{r}{\pi}
& X
\end{tikzcd}
\end{center}

 As $\pi\circ s=id_X$, we get a bundle homomorphism $\pi^*\circ d''s: TX\lra E$ over $id:X\lra X$. Define \begin{align}\label{partial} \bar{\p}_Es=\pi^*\circ d''s|_{(TX)^{(0,1)}},\end{align} then $\bar{\p}_Es\in \G(X,(T^*X)^{(0,1)}\otimes E)$. To verify that $\bar{\p}_E$ satisfies the Leibniz rule, from $d(fs)=df s+fds$, we get
 \begin{align*} d''(fs)&=\f{1}{2}(d(fs)+ \mathcal{J}\circ d(fs)\circ J)=\f{1}{2}(df s+fds+\mathcal{J}\circ (dfs+fds)\circ J)\\
 &=\f{1}{2}((df+idf\circ J)s+f(ds+ \mathcal{J}\circ ds\circ J))\\
 &=\bar{\p}fs+fd''s
 \end{align*}
 where we use property ii) and iii) of $\mathcal{J}$ in the second line. Passing to $E$, we have $\bar{\partial}_E (fs)=\bar{\partial}f\otimes s+f\bar{\partial}_Es$. So $\bar{\partial}_E$ gives a pseudoholomorphic structure on $E$.

On the other hand, let $\bar{\p}_E$ be a pseudoholomorphic structure and $s:X\lra E$ in $\G(X,E)$. Then $\bar{\p}_Es\in\G(X,(T^*X)^{(0,1)}\otimes E)$ and $\overline{\bar{\p}_Es}\in\G(X,(T^*X)^{(1,0)}\otimes E)$. Together they give $\bar{\p}_Es+\overline{\bar{\p}_Es}: TX\lra E$. As $ V(E)\cong\pi^*(E)=\{(v_1,v_2)\in E\times E, \pi(v_1)=\pi(v_2)\}$, there is an embedding $s^*: E\lra V(E)$ induced by $s$ given by $$s^*(v)=(s\pi(v), v).$$
 Composing it with $\bar{\p}_Es+\overline{\bar{\p}_Es}$, we have $s^*\circ (\bar{\p}_Es+\overline{\bar{\p}_Es}): TX\lra V(E)\subset TE$. At each point $v$ of $E$, from $d\pi\circ ds=id_{TX}$, we get $T_vE=ds(TX)\oplus V_v(E)$. Then define the bundle almost complex structure $\mathcal{J}$ on $TE$ by letting $$\mathcal{J}=(-2s^*\circ (\bar{\p}_Es+\overline{\bar{\p}_Es})+ds)\circ J\circ ds^{-1}$$ on $ds(TX)$ and $\mathcal{J}=J_{st}$ on the vertical tangent space where $J_{st}$ is the standard complex structure by multiplying with $i$. Using the Leibniz rule of $\bar{\p}_E$, it can be showed that $\mathcal{J}$ is in independent of the smooth sections $s$ and satisfies properties i), ii), iii). Also, the constructions give the one-to-one correspondence discovered in \cite{DeT}.
\end{proof}

Denote $\bar{\p}_{\mathcal{J}}$ the pseudoholomorphic structure determined by a bundle almost complex structure $\mathcal{J}$. We have

\begin{cor} \label{holo}
For any $s\in \Gamma(X, E)$, $\bar{\p}_{\mathcal{J}}s=0$ if and only if $s$ is $(J, \mathcal{J})$ holomorphic.
\end{cor}

\begin{proof}
Since the result is used frequently in this note, we offer two separate proofs which have their own ingredients.

The first proof follows directly from Proposition \ref{deT}. From (\ref{partial}) we have $\pi^*\circ d's=\bar{\p}_{\mathcal{J}}s+\overline{\bar{\p}_{\mathcal{J}}s}$. Then $\bar{\p}_{\mathcal{J}}s=0$ is equivalent to $d's=0$, which means that $s$ is $(J, \mathcal{J})$ holomorphic.

The second proof applies a local argument. For each 2-dimensional $J$-invariant subspace $P$ in $T_pX$ at a point $p$, we know there is a $J$-holomorphic disk $D$ passing through $p$ with the tangent plane $P$. Then $J$ is integrable on $D$ and $E|_D$ is a holomorphic bundle by dimension reason (see the argument after Definition 3.2). Restricted on $D$, it is known that $\bar{\p}_{\mathcal{J}}s=0$ is equivalent to that $s$ is $(J, \mathcal{J})$ holomorphic. Since both $\bar{\p}_{\mathcal{J}}s=0$ and $(J, \mathcal{J})$ holomorphic are local conditions and only depend on the complex directions, we can choose $P$ to be any directions and obtain the general equivalence.
\end{proof}

We call such a section $s$ a pseudoholomorphic section of $(E, \mathcal J)$. The above correspondence builds the bridge to the second author's paper \cite{Z2} on the intersections of almost complex submanifolds, and is used frequently in this paper. In particular, when $E$ is a complex line bundle over $4$-manifold $(X, J)$, the zero locus of a pseudoholomorphic section $s$ is a $J$-holomorphic $1$-subvariety in class $c_1(E)$ by Corollary 1.3 of \cite{Z2}.

For any $(E, \mathcal J)$, define $$H^0(X, (E, \mathcal J))=H^{0,0}_{\bar{\p}_{\mathcal J}}(X, E)=\{s\in \Gamma(X, E): \bar{\p}_{\mathcal J} s=0\}.$$ By Theorem \ref{finite}, it is finite dimensional. The $(E, \mathcal J)$-genus of $X$ is defined as $P_{E, \mathcal J}:=\dim H^0(X, (E, \mathcal J))$. When there is no confusion of the choice of bundle almost complex structure $\mathcal J$, we will simply write it as $P_E$. The bundle almost complex structure $\mathcal J$ on $E$ induces bundle almost complex structures on $E^{\otimes m}$, which is also denoted by $\mathcal J$. Thus the notation $P_{E^{\otimes m}, \mathcal J}$ or simply $P_{E^{\otimes m}}$ makes sense. When $E=\mathcal K$ endowed with the standard bundle almost complex structure, we have the $m^{th}$ plurigenus of $(X,J)$ is defined to be $P_m(X, J)=\dim H^0(X, \mathcal K^{\otimes m})$.

We are now ready to define the Iitaka dimension (and the Kodaira dimension).

\begin{defn}
The Iitaka dimension $\kappa^J(X, (L, \mathcal J))$ of a complex line bundle $L$ with bundle almost complex structure $\mathcal J$ over $(X, J)$ is defined as
$$\kappa^J(X, (L, \mathcal J))=\begin{cases}\begin{array}{cl} -\infty, &\ \text{if} \ P_{L^{\otimes m},\mathcal J}=0\  \text{for any} \ m\geq 0\\
\limsup_{m\rightarrow \infty} \dfrac{\log P_{L^{\otimes m},\mathcal J}}{\log m}, &\  \text{otherwise.}
\end{array}\end{cases}$$

The Kodaira dimension $\kappa^J(X)$ is defined by choosing $L=\mathcal K$ and $\mathcal J$ to be the bundle almost complex structure induced by $\bar{\p}$.
\end{defn}

\section{Birational invariants}\label{birational}

As suggested by the results in \cite{Z2}, degree $1$ pseudoholomorphic maps are the right notion of birational morphism in almost complex category. We define two almost complex manifolds $M$ and $N$ to be birational to each other if there are almost complex manifolds $M_1, \cdots, M_{n+1}$ and $X_1, \cdots, X_{n}$ such that $M_1=M$ and $M_{n+1}=N$, and there are degree one pseudoholomorphic maps $f_i: X_i\rightarrow M_i$ and $g_i: X_{i}\rightarrow M_{i+1}$, $i=1, \cdots, n$.

The next natural step is to find birational invariants. In birational geometry, there are many important birational invariants, including the fundamental group, the Hodge numbers $h^{p,0}$, the plurigenera and in particular the Kodaira dimension. As shown in Theorem 1.5 of \cite{Z2}, $X=M\#k\overline{\mathbb CP^2}$ when there is a degree one pseudoholomorphic map $\phi: X\rightarrow M$ between $4$-dimensional almost complex manifolds. Hence, the fundamental group is apparently also birationally invariant in the almost complex category.

We will show in this section that the almost complex Kodaira dimension $\kappa^J$, plurigenera $P_m$ and Hodge numbers $h^{p,0}$ are birational invariants for $4$-dimensional almost complex manifolds.

We first show that plurigenera and Kodaira dimension for almost complex $4$-manifolds are non-increasing under pseudoholomorphic maps of non-zero degree.

\begin{lem}\label{kjsurj}
Let $u: (X, J)\rightarrow (Y, J_Y)$ be a surjective pseudoholomorphic map between closed almost complex $2n$-manifolds. Then $P_m(X, J)\ge P_m(Y, J_Y)$. Hence, $\kappa^{J}(X)\ge \kappa^{J_Y}(Y)$.
\end{lem}
\begin{proof}
Pullback of sections defines $$u^*_m: H^0(Y, \mathcal K_Y^{\otimes m})\rightarrow H^0(X, \mathcal K_X^{\otimes m})$$ for all $m\ge 1$. Combining the argument of Theorem 5.5 and the result of Theorem 3.8 in \cite{Z2}, we know that the  singularity subset $\mathcal S_u$ has finite $(2n-2)$-dimensional Hausdorff measure. (Theorem 1.4 of \cite{Z2} shows that  $\mathcal S_u$ supports a $J$-holomorphic $1$-subvariety when $n=2$.)

For any $s\in H^0(Y, \mathcal K_Y^{\otimes m})$, if $u^*_m(s)=0$, then the restriction $u_m^*(s)|_{X\setminus \mathcal S_u}=0$ would imply $s|_{Y\setminus u(\mathcal S_u)}=0$. Since $s$ is smooth and $\overline{Y\setminus u(\mathcal S_u)}=Y$, we know $s=0$. Hence $u_m^*$ is injective, which implies the inequalities.
\end{proof}

For any pseudoholomorphic structure $\bar\partial_E$ of a complex vector bundle $E$, it also induces a pseudoholomorphic structure on $E|_D$ for any non-compact embedded $J$-holomorphic curve $D\subset X$. By Koszul-Malgrange theorem, it is holomorphic. Since $D$ is Stein, by Oka's principle, any holomorphic bundle is isomorphic to the product $D\times \C^k$.

Let $s$ be a smooth section of $E$ over a compact almost complex manifold $X$. Then for any point $x\in X$ and a $J$-holomorphic disk $D$ passing through it, we could write $s|_{D}$ as a vector valued complex function $s': D\rightarrow \C^k$. In fact, $s'$ is the composition of $s$ with the projection from $D\times \C^k$ to $\C^k$. Since the  projection is holomorphic, $s$ is $(J, \mathcal J)$-holomorphic if and only if $s'$ is holomorphic.  In other words, $s'$ is a holomorphic function. Later, we will simply write $s$ instead of $s'$ by abuse of notation.

Since there is  no local complex coordinate system for a general almost complex manifold, we use the $J$-fiber-diffeomorphism \cite{T96} to play such a role.

We start with any point $x\in M$, and want to choose an open neighborhood $U$ of $x$. Without loss of generality, as in \cite{T96, Z2}, we can assume the almost complex structure $J$ is on $\mathbb C^2$. It agrees with the standard almost complex structure $J_0$ at the origin, but typically nowhere else.

Denote a family of holomorphic disks $D_w:=\{(\xi, w)| |\xi|<\rho\}$, where $w\in D$. What we get from \cite{T96}, mainly Lemma 5.4, is a diffeomorphism $f: D\times D \rightarrow \mathbb C^2$ onto its image $U$, where $D\subset \mathbb C$ is the disk of radius $\rho$, such that:
\begin{itemize}
 \item For all $w\in D$, $f(D_w)$ is a $J$-holomorphic submanifold containing $(0, w)$.
 \item For all $w\in D$, dist$((\xi, w); f(\xi, w))\le z\cdot \rho\cdot |\xi|$. Here $z$ depends only on $\Omega$ and $J$.
\item For all $w\in D$, the derivatives of order $m$ of $f$ are bounded by $z_m\cdot \rho$, where $z_m$ depends only on $\Omega$ and $J$.
\end{itemize}

We call such a diffeomorphism $J$-fiber-diffeomorphism. We have freedom to choose the ``direction" of these disks by rotating the Gaussian coordinate system. As in \cite{T96, Z2, BZ}, we are also able to choose the center $f(0\times D)$ to be $J$-holomorphic.

With these preparations, we are able to derive the following version of Hartogs' extension theorem for almost complex manifolds.

\begin{thm}\label{hartogs}
Let $(E, \mathcal J)$ be a complex vector bundle with a bundle almost complex structure over the almost complex $4$-manifold $(X, J)$, and $p\in X$. Then any section in $H^0(X\setminus p, (E, \mathcal J)|_{X\setminus p})$ extends to a section in $H^0(X, (E, \mathcal J))$.
\end{thm}
\begin{proof}
Near $p$, as in \cite{Z2}, we choose a $J$-fiber-diffeomorphism of a neighborhood $U$ of $p$, $f: D\times D\rightarrow U$, such that $f(0\times D)$ and $f(D\times w), \forall w\in D$ are embedded $J$-holomorphic disks. By possibly shrinking $U$, our complex vector bundle $(E, \mathcal J)$ could be trivialized such that each section of it (on a subset of $U$) restricts to $f(D_w)$ and $f(0\times D)$ are complex vector valued functions. We can achieve it by first choose the trivialization along $f(0\times D)$ and then fiberwise along each $f(D_w)$.

Let $s\in H^0(X\setminus p, (E, \mathcal J)|_{X\setminus p})$. By choosing the above trivialization and the previous discussion, $s$ is a vector valued holomorphic function along each $D_w$ when $w\ne 0$, $(D\setminus \{0\})\times 0$ and $0\times (D\setminus \{0\})$. We use Cauchy integration formula to define  $$a_j(z_2)=\frac{1}{2\pi i}\int_{|\xi|=\rho}\frac{s(\xi, z_2)}{\xi^{j+1}}d\xi, \, \,\, \forall j\in \mathbb Z.$$ It is a smooth  (vector valued) function and $a_0(z_2)=s(0, z_2)$ when $z_2\ne 0$. Hence, in particular, $a_0(z_2)$ is holomorphic on $D\setminus \{0\}$. We let $a_j(0)=\lim_{z_2\rightarrow 0} a_j(z_2)$. Since for fixed $z_2\ne 0$, $ s(\xi, z_2)$ is holomorphic for $\xi \in D$, we know $a_{-j}(z_2)=0$ for $j>0$. By the continuity of $s$, we know $a_{-j}(0)=0, \forall j>0$. Hence, $s(\xi, 0)=\sum_{j=-\infty}^{\infty}a_j(0) \xi^j=\sum_{j=0}^{\infty}a_j(0) \xi^j$ is also holomorphic at $\xi=0$ with value $a_0(0)$ at $\xi=0$. In particular, $a_0(0)=\frac{1}{2\pi i}\int_{|\xi|=\rho}\frac{s(\xi, 0)}{\xi}d\xi$. Since $a_0(z_2)=s(0, z_2)$ is holomorphic when $z_2\ne 0$, the partial derivative $$\frac{\partial}{\partial \bar z_2}a_0(z_2)=\frac{1}{2\pi i}\int_{|\xi|=\rho}\frac{\frac{\partial}{\partial \bar z_2}s(\xi, z_2)}{\xi}d\xi=0.$$ This extends to $z_2= 0$ since $s$ is smooth. Hence, $a_0(z_2)=s(0, z_2)$ is also holomorphic at $z_2=0$.

To summarize, what we have proved in the above is that the extensions of holomorphic functions $s(0, z_2)$ and $s(z_1, 0)$ on $0\times (D\setminus\{0\})$ and $(D\setminus\{0\})\times 0$ to $(0, 0)$  have the same value $a_0(0)$, and are holomorphic at both disks $0\times D$ and $D_0$.  As in \cite{Z2}, we can choose the center $f(0\times D)$ of the $J$-fiber diffeomorphism  (that transverse to $f(D_0)$) to be (a subdisk of) any given $J$-holomorphic disk. Moreover, we can also choose a family of disks passing through $p$ whose complex tangent directions at $p$ form a disk around a given direction $\kappa$ in $\mathbb CP^1$. Moreover, each of them is the $D_0$ fiber of a $J$-fiber diffeomorphism (see Lemma 5.8 of \cite{T96} or Lemma 3.10 of \cite{Z2}).  Since $\mathbb CP^1$ is compact, we can choose finite many such families such that their union covers a neighborhood of $p$, and their tangent directions cover $\mathbb CP^1$.

We choose $J$-fiber diffeomorphisms around $p$ whose fiber passing through $p$ lying in the above union of families and take the center of the foliation to be either $f(0\times D)$ or $f(D_0)$ in our above construction. By this process, we know all the disks in the above union of families have the same extended value at $p$, and are holomorphic at all the directions. Hence, our section $s\in H^0(X\setminus p, (E, \mathcal J)|_{X\setminus p})$ is extended over $p$ to a section in $H^0(X, (E, \mathcal J))$.
\end{proof}

We are ready to show the Kodaira dimension $\kappa^J$ is a birational invariant for almost complex $4$-manifolds.

\begin{thm}\label{Kodbir}
Let $u: (X, J)\rightarrow (Y, J_Y)$ be a degree one pseudoholomorphic map between closed almost complex $4$-manifolds. Then $P_m(X, J)=P_m(Y, J_Y)$ and thus $\kappa^{J}(X)= \kappa^{J_Y}(Y)$.
\end{thm}
\begin{proof}
First, by Corollary \ref{holo}, we know that any element in $H^0(X, \mathcal K_X^{\otimes m})$ is $(J, \mathcal J_{J})$-holomorphic where $\mathcal J_J$ is the bundle almost complex structure corresponding to $\bar{\p}_m$.

By Lemma \ref{kjsurj}, we only need to show $u^*_m$ is surjective.
By Theorem 1.5 of \cite{Z2}, we know there is a finite set $Y_1\subset Y$ such that $$u: X\setminus u^{-1}(Y_1) \rightarrow Y\setminus Y_1$$ is a diffeomorphism.  For $\sigma\in H^0(X, \mathcal K_X^{\otimes m})$, we could pull it back by $u^{-1}|_{Y\setminus Y_1}$ to get $(u^{-1})^*(\sigma)\in H^0(Y\setminus Y_1, \mathcal K_{Y\setminus Y_1}^{\otimes m})$. By Theorem \ref{hartogs}, we could extend point-by-point over $Y_1$ to get a unique element in $H^0(Y, \mathcal K_Y^{\otimes m})$.

 Hence, $u_m^*$ is surjective and we complete the proof.
\end{proof}

There are also other birational invariants. Since $\Lambda^{p,1}\cong (T^*X)^{0,1}\otimes \Lambda^{p,0}$, the $\bar{\p}$ operator induces a natural pseudoholomorphic structure on $\Lambda^{p,0}$ for $0\leq p\leq n$. By Proposition \ref{finsecE}, $H^0(X, \Omega^p_X):=H^0(X, \Omega^p(\mathcal O))=\mathcal H_{\bar{\p}}^{(p,0)}(X, \mathcal O)$ is finite dimensional. We denote $h^{p,0}(X):=\dim H^0(X, \Omega^p(\mathcal O))$. For a pseudoholomorphic map $u: (X, J)\rightarrow (Y, J_Y)$ between closed almost complex $2n$-manifolds and any $0\le p\le n$, pullback of sections defines $u^*: H^0(Y, \Omega^p_Y)\rightarrow H^0(X, \Omega^p_X)$. When $u$ is surjective, by the same argument as Lemma \ref{kjsurj}, we have

\begin{lem}\label{p0surj}
Let $u: (X, J)\rightarrow (Y, J_Y)$ be a surjective pseudoholomorphic map between closed almost complex $2n$-manifolds. Then $u^*: H^0(Y, \Omega^p_Y)\rightarrow H^0(X, \Omega^p_X)$ is injective and $h^{p,0}(X)\ge h^{p,0}(Y)$ for any $0\le p\le n$.
\end{lem}

We can also show that $h^{p,0}$ are birational invariants in dimension $4$. In fact, the only one which does not follow from Theorem \ref{Kodbir} is the irregularity $h^{1,0}$.

\begin{thm}\label{hodgebir}
Let $u: (X, J)\rightarrow (Y, J_Y)$ be a degree one pseudoholomorphic map between closed almost complex $4$-manifolds. Then $h^{p,0}(X)=h^{p,0}(Y)$ for any $0\le p\le 2$.
\end{thm}
\begin{proof}
First, by Corollary \ref{holo}, we know that any element in $H^0(X, \Omega^p_X)$ is $(J, \mathcal J_{J})$-holomorphic where $\mathcal J_J$ is the bundle almost complex structure on $\Lambda^{p,0}$ corresponding to he natural pseudoholmorphic structure induced by $\bar{\p}$.

By Lemma \ref{p0surj}, we only need to show $u^*$ is surjective.
By Theorem 1.5 of \cite{Z2}, we know there is a finite set $Y_1\subset Y$ such that $$u: X\setminus u^{-1}(Y_1) \rightarrow Y\setminus Y_1$$ is a diffeomorphism.  For $\sigma\in H^0(X, \Omega^p_X)$, we could pull it back by $u^{-1}|_{Y\setminus Y_1}$ to get $(u^{-1})^*(\sigma)\in H^0(Y\setminus Y_1, \Omega^p_{Y\setminus Y_1})$, which are the pseudoholomorphic sections of $\Lambda^{p,0}(Y)$ over $Y\setminus Y_1$. By Theorem \ref{hartogs}, we could extend it point-by-point over $Y_1$ to get a unique element in $H^0(Y, \Omega^p_Y)$.
 Hence, $u^*$ is surjective and we complete the proof.
\end{proof}
We would like to remark that the dimension of the $J$-anti-invariant cohomology $H_J^-(X, \mathbb R)$ defined in \cite{LZcag} is also a birational invariant as shown in \cite{BZ}.

\section{Examples}\label{exam}
In this section, we give some explicit examples on the calculation of the almost complex plurigenera, the Kodaira dimension $\kappa^J$ and the irregularity. As we have seen in Section \ref{birational}, all of them are birational invariant on 4-manifolds.  However, different from the integrable case, they are no longer deformation invariants. This is easy to see by deforming an integrable almost complex structure of a surface of general type as shown in the introduction. This argument does not quite extend to the case when the canonical class is torsion. Our first two examples study such explicit deformations on Kodaira-Thurston surface and $4$-torus.

In Section \ref{bigex}, we show that there are examples of compact $2n$-dimensional nonintegrable almost complex manifolds with Kodaira dimension $\{-\infty, 0, 1, \cdots, n-1\}$ for $n\geq 2$ (Theorem \ref{niKod}).

\subsection{The Kodaira-Thurston surface}\label{kt}
Consider the Kodaira-Thurston surface $X=S^1\times (\Gamma \backslash \text{Nil}^3)$, where $\text{Nil}^3$ is the Heisenberg group
 $$\text{Nil}^3=\{ A\in GL(3, \mathbb{R})| A=\begin{pmatrix} 1&x&z\\0&1&y\\0&0&1\end{pmatrix}\}$$
 and $\Gamma$ is the subgroup in $\text{Nil}^3$ consisting of element with integer entries, acting by left multiplication (see \cite{TW}). $X$ is homogeneous and has trivial tangent and cotangent bundle. An invariant frame of the tangent bundle is given by $$\frac{\p}{\p t}, \hspace{4mm} \frac{\p}{\p x},\hspace{4mm} \frac{\p}{\p y}+x\frac{\p}{\p z}, \hspace{4mm}\frac{\p}{\p z},$$
 where $t$ is the coordinate of $S^1$. The corresponding dual invariant coframe is given by
 $$dt, \hspace{4mm}dx,\hspace{4mm} dy,\hspace{4mm} dz-xdy.$$
For any $a\neq 0\in \mathbb{R}$, define the almost complex structures $J_a$ by:
$$J_a(\frac{\p}{\p t})=\frac{\p}{\p x}, \hspace{2mm} J_a(\frac{\p}{\p x})=-\frac{\p}{\p t}, \hspace{2mm} J_a(\frac{\p}{\p y}+x\frac{\p}{\p z})=\frac{1}{a}\frac{\p}{\p z},\hspace{2mm} J_a(\frac{\p}{\p z})=-a(\frac{\p}{\p y}+x\frac{\p}{\p z}).$$
We can compute the Nijenhuis tensor to get $N(\frac{\p}{\p x}, \frac{\p}{\p z})=a^2(\frac{\p}{\p y}+x\frac{\p}{\p z})\neq 0$. Therefore $J_a$ is not integrable by the Newlander-Nirenberg theorem \cite{NN}. As $(T^*X)^{1,0}$ is spanned by $\phi_1=dt+idx, \phi_2=dy+ia(dz-xdy)$, any section of $\mathcal K$ can be written as $s=f\phi_1\wedge\phi_2$. Since $$d\phi_2=-iadx\wedge dy=-\frac{a}{4}(\phi_1\wedge\bar{\phi}_2-\bar{\phi}_1\wedge \phi_2+\bar{\phi}_1\wedge\bar{\phi}_2-\phi_1\wedge\phi_2),$$ we have
\begin{align*}
\bar{\p}(\phi_1\wedge\phi_2)&=-\phi_1\wedge\bar{\p}\phi_2=\frac{a}{4}\phi_1\wedge(\phi_1\wedge\bar{\phi}_2-\bar{\phi}_1\wedge\phi_2)\\
&=\frac{a}{4}\bar{\phi}_1\wedge\phi_1\wedge\phi_2.
\end{align*}
So $\bar{\p}s=0$ if and only if \begin{align}\bar{\p}f+\frac{a}{4}f\bar{\phi}_1=0. \label{09}\end{align}
Let $w=t+ix, \frac{\p}{\p \bar{w}}=\frac{1}{2}(\frac{\p}{\p t}+i\frac{\p}{\p x})$ and $V=\frac{1}{2}((\frac{\p}{\p y}+x\frac{\p}{\p z})+i\frac{1}{a}\frac{\p}{\p z})$. Then $\frac{\p}{\p \bar{w}}, V$ are dual vectors of $\bar{\phi}_1, \bar{\phi}_2$. From (\ref{09}) we have
\begin{align}\label{04}&\frac{\p f}{\p \bar{w}}+\frac{a}{4}f=0\\
\label{5}&V(f)=0.
\end{align}
Let $f=f_1+if_2$, where $f_1, f_2$ are smooth real functions on $X$. From (\ref{5}) we get that $\bar{V}Vf=0$ where $\bar{V}$ is the conjugate of $V$. As $\bar{V}V=\f{1}{4}((\frac{\p}{\p y}+x\frac{\p}{\p z})^2+(\frac{1}{a}\frac{\p}{\p z})^2)$, we obtain
\begin{align}\label{8}
 &\frac{\p^2 f_1}{\p y^2}+2x\frac{\p^2f_1}{\p y\p z}+(x^2+\f{1}{a^2})\frac{\p^2f_1}{\p z^2}=0,\\
\label{9}
&\frac{\p^2 f_2}{\p y^2}+2x\frac{\p^2f_2}{\p y\p z}+(x^2+\f{1}{a^2})\frac{\p^2f_2}{\p z^2}=0
\end{align}
Consider the fibration $\rho: X\rightarrow T^2=\mathbb{R}^2/\mathbb{Z}^2$ given by
$$\rho([t, x, y, z])=[t, x].$$
The fiber of $\rho$ is a torus with coordinate $(y,z)$. \eqref{8}, \eqref{9} is strictly elliptic without zero order term when viewing $f$ as a function of $y, z$. As the fiber is compact, by the maximum principle $f$ is constant in each fiber. We can push down $f$ to a function on the base $T^2$ with $(t, x)$ coordinate. To solve the equation (\ref{04}) on $T^2$, consider the Fourier series $$\mathcal{F}(f)=\sum_{(k,l)\in \mathbb{Z}^2}f_{k,l}e^{2\pi i(kt+lx)},\ f_{k,l}=\int_{T^2} f(t,x)e^{-2\pi i(kt+lx)}dtdx.$$ For smooth function $f$, $f=0$ if and only if $f_{k, l}=0, \forall (k,l)\in \mathbb Z^2$ by the completion of the series $\{e^{2\pi i(kt+lx)}\}$.
Apply $\mathcal{F}$ to \eqref{04}, we get
$$\sum_{(k,l)\in \mathbb{Z}^2} (\frac{a}{4}+\pi(ik-l))f_{k,l} e^{2\pi i(kt+lx)}=0.$$
If $a\notin 4\pi\mathbb{Z}$, then $\frac{a}{4}+\pi(k-il)\neq 0$ for any $(k,l)\in \mathbb{Z}^2$. So $f_{k,l}=0$ and $f=0$. If $a=4l\pi$ for some $l\in\mathbb{Z}\backslash \{0\}$, then $f=Ce^{2\pi i lx}$ are the solutions. Therefore we get
$$P_1(X, J_a)=\begin{cases} 0, a\notin 4\pi\mathbb{Z}\\
1, a\in 4\pi\mathbb{Z}\end{cases}.$$
For $m\geq 2$, assume that $s=f(\phi_1\wedge\phi_2)^{\otimes m}$ is a holomorphic section of $\mathcal K^{\otimes m}$. Then $$\bar{\p}_m s=(\bar{\p}f + \frac{ma}{4}f\bar{\phi}_1)(\phi_1\wedge\phi_2)^{\otimes m}=0.$$ The same computation from above shows that $f$ is constant on $(y,z)$ and satisfying
$$\frac{\p f}{\p \bar{w}}+\frac{ma}{4}f=0.$$
Using Fourier transform, we get that if $a\notin \frac{4}{m}\pi\mathbb{Z}$, then $f=0$; if $a=\frac{4l\pi}{m}$ for some $l\in \mathbb{Z}\backslash \{0\}$, then $f=Ce^{2\pi i lx}$. So
\begin{equation}\label{KTpm}
P_m(X, J_a)=\begin{cases} 0, a\notin \frac{4}{m}\pi\mathbb{Z}\\
1, a\in \frac{4}{m}\pi\mathbb{Z}.\end{cases}
\end{equation}

\vspace{.2cm}
To compute the irregularity $h^{1,0}(X)$, assume that $\gamma=g_1\phi_1+g_2\phi_2\in H^0(X,\Omega_X)$. As $d\phi_1=0$ and $\bar{\p}\phi_2=-\frac{a}{4}\phi_1\wedge\bar{\phi}_2+\frac{a}{4}\bar{\phi}_1\wedge \phi_2$, from $\bar{\p}\gamma=0$ we obtain
\begin{align}\bar{\p}g_1+\frac{a}{4}g_2\bar{\phi}_2=0 \label{18} \\ \bar{\p}g_2-\frac{a}{4}g_2\bar{\phi}_1=0 \label{19} \end{align}
(\ref{19}) is in the same form with \eqref{09}. So we have $V(g_2)=0$ and then $g_2$ is independent of $y,z$. (\ref{18}) is equivalent to
 \begin{align}
  \label{020} \frac{\p g_1}{\p \bar{w}}=0\\V(g_1)+\frac{a}{4}g_2=0  \label{021}
 \end{align}
 From \eqref{020} we have that $g_1$ is independent of $t, x$. As $g_2$ is independent of $y,z$, composing $\bar{V}$ to (\ref{021}) we get that $\bar{V}V(g_1)=0$. Therefore, $g_1$ is a constant. Returning to (\ref{021}) we get that $g_2=0$. Therefore $\gamma=c\phi_1$ for some constant $c$ and $h^{1,0}(X)=1$.

In conclusion, we have

\begin{prop}\label{6.1}
For any $a\neq 0 \in \mathbb{R}$, there is a nonintegrable almost complex structure $J_a$ on $X=S^1\times (\Gamma \backslash \text{Nil}^3)$ such that $h^{1,0}(X)=1$ and
$$\kappa^{J_a}(X)=\begin{cases}\begin{array}{cl} -\infty, &\ a\notin \pi\mathbb{Q}\\
0, &\ a\in \pi\mathbb{Q}\backslash\{0\}\end{array}\end{cases}$$
\end{prop}

\begin{proof}

As $\bigcup_{m\in\mathbb{Z}_+} \frac{4}{m}\pi\mathbb{Z}=\pi\mathbb{Q}$, we have that if $a\notin \pi\mathbb{Q}$, $P_m=0$ for all $m$; if $a\in \pi\mathbb{Q}$, then $P_m=1$ for some $m$.
\end{proof}

If we choose $a=4\pi, 2\pi, \frac{4}{3}\pi, \cdots, \frac{4}{n}\pi, \cdots$, then the first nonzero plurigenera are $P_1, P_2, P_3, \cdots, P_n,\cdots$. Therefore they are not birationally equivalent though with $\kappa^J=0$.

\begin{remk}
Let $J$ be the almost complex structure given by
$$J(\frac{\p}{\p t})=\frac{\p}{\p z}, \hspace{2mm} J(\frac{\p}{\p x})=\frac{\p}{\p y}+x\frac{\p}{\p z}.$$
Then $J$ is integrable and induces the usual complex structure on $X$. In this case, $\mathcal K$ is holomorphically trivial with a closed section $(dt+i(dz-xdy))\wedge(dy+idx)$. So $P_m(X,J)=1$ for any $m\geq 1$ and $\kappa^{J}(X)=0$.
\end{remk}

From \eqref{KTpm}, we see that both plurigenera and Kodaira dimension are not deformation invariant.  However, we still have upper semi-continuity.
Assume that $\Delta$ is an open set in $\mathbb{C}$ and $\{J(t), t\in \Delta\}$ is a family of almost complex structures on a compact smooth manifold, depending smoothly on $t$. Let $P_m(t), h^{p,0}(t)$ be the $m$-th plurigenus and $(p,0)$ Hodge number of $J(t)$. We have
\begin{prop} $P_m(t)$ and $h^{p,0}(t)$ are upper semi-continuous function of $t$.
\end{prop}
\begin{proof}
As all sections in $H^0(X,\mathcal K(t)^{\otimes m})$ and $H^0(X,\Omega^p(t))$ are exactly the harmonic sections, by the properties of elliptic operators (Theorem 4.3 in \cite{KM}, see also \cite{DLZ}), $P_m(t)$ and $h^{p,0}(t)$ are upper semi-continuous.
\end{proof}

\subsection{$4$-torus}\label{T4}
We offer another example on the four torus.
Consider the four torus $X=T^4=\mathbb{R}^4/ \mathbb{Z}^4$ with coordinates $(x_1, x_2, x_3, x_4)$. We study the almost complex structure $J$ introduced in \cite{CKT} given by
$$J=\begin{pmatrix} 0&-1&\a&\b\\1&0&-\b&\a\\0&0&0&1\\0&0&-1&0\end{pmatrix}.$$
We assume that $\a, \b$ are any two real smooth functions on $T^4$ satisfying $\f{\p^2 (\b+i\a)}{\p x_1^2}+\f{\p^2 (\b+i\a)}{\p x_2^2}\neq 0$ in a dense open set. For example, $\a=\cos 2\pi(x_1+x_2), \b=\sin 2\pi(x_1+x_2)$. Direct computation shows that $J$ is integrable if and only if $\a,\b$ are independent of $x_1,x_2$ (see \cite{CKT}). Therefore, $J$ is not integrable by our assumption. Let $$\phi_1=dx_1+i(dx_2-\a dx_3-\b dx_4),\ \ \ \phi_2=dx_3-idx_4.$$
Then $(T^*X)^{1,0}$ is spanned by $\phi_1, \phi_2$. Assume that $s=f\phi_1\wedge\phi_2$ is a smooth section of $\mathcal K$. Let $w=x_1+ix_2, \frac{\p}{\p w}=\frac{1}{2}(\frac{\p}{\p x_1}-i\frac{\p}{\p x_2})$, we have $\bar{\p}s=0$ if and only if
\begin{align}\label{t4}
\bar{\p}f+\frac{1}{2}\dfrac{\p(\b+i\a)}{\p w} f\bar{\phi}_2=0
\end{align}
It is equivalent to
\begin{align}\label{t41}
&\f{\p f}{\p \bar{w}}=0\\
\label{t42}\f{\p f}{\p x_3}+\a\f{\p f}{\p x_2}-i(&\f{\p f}{\p x_4}+\b\f{\p f}{\p x_2})+\dfrac{\p(\b+i\a)}{\p w} f=0
\end{align}
As $T^4$ is compact, from \eqref{t41} we get that $f$ is constant in the $(x_1,x_2)$ direction. Then \eqref{t42} become
\begin{align}\label{t42'} \f{\p f}{\p x_3}-i\f{\p f}{\p x_4}+\dfrac{\p(\b+i\a)}{\p w} f=0.\end{align}
Apply $\f{\p}{\p \bar{w}}$ to \eqref{t42'} to get
\begin{align}\dfrac{\p^2(\b+i\a)}{\p w\p\bar{w}} f=0.\end{align}
By the assumption of $\a, \b$, we have $f=0$ and $s=0$. Similarly, for $\mathcal K^{\otimes m}, m\geq 2$, if $s=f(\phi_1\wedge\phi_2)^{\otimes m}$ is holomorphic, then
$$\bar{\p}f+\frac{m}{2}\dfrac{\p(\b+i\a)}{\p w} f\bar{\phi}_2=0.$$
The same argument gives that $s=0$. Therefore, $P_m(X,J)=0, m\geq 1$ and $\kappa^J(X)=-\infty$.\\

For the irregularity, assume that $\gamma=g_1\phi_1+g_2\phi_2\in H^0(X,\Omega_X)$. Then from $\bar{\p} \gamma=0$ we get that \begin{align} \bar{\p}g_1+\frac{1}{2}\dfrac{\p(\b+i\a)}{\p w} g_1\bar{\phi}_2=0 \label{21}\\ \bar{\p}g_2+\frac{1}{2}\dfrac{\p(\b-i\a)}{\p \bar{w}} g_1\bar{\phi}_1=0. \label{22}\end{align}
(\ref{21}) is the same with (\ref{t4}), so we get that $g_1=0$. Putting it to (\ref{22}), we deduce that $g_2$ is a constant. So $\gamma=c\phi_2$ and $h^{1,0}=1$.

For any $t=t_1+it_2\in \mathbb{C}$, let $$J(t)=\begin{pmatrix} 0&-1&t_1\a&t_2\b\\1&0&-t_2\b&t_1\a\\0&0&0&1\\0&0&-1&0\end{pmatrix}.$$
$J(0)$ is the standard complex structure on $T^4$ and $J(1+i)=J$. By the above calculation, for any $m\geq 1$, $t\in \mathbb{C}$, we have $$P_m(t)=\begin{cases} 0, t\neq 0 \\ 1, t=0\end{cases}, \, \,\, h^{1,0}(t)=\begin{cases} 1, t\neq 0 \\ 2, t=0\end{cases}.$$

This gives an example where the plurigenera, the Kodaira dimension and the irregularity are not constant under smooth deformation even when $K=0$.

\subsection{Non-integrable almost complex manifolds with large Kodaira dimension}\label{bigex}
Although a generic almost complex structure does not have any pseudoholomorphic curve, which forces Kodaira dimension to be $-\infty$ or $0$, we still have interesting non-integrable examples with large Kodaira dimension.
In this subsection, we give examples of non-integrable almost complex structures on $2n$-manifolds with Kodaira dimension lying among $-\infty, 0, 1, \cdots, n-1$. First, we construct non-integrable almost complex 4-manifolds with $\kappa^J=1$.

Let $S$ be a compact Riemann surface with genus $g\geq 2$. We shall define a nonintegrable almost complex structure on $X=T^2\times S$.\\
Denote the two projections by
$$\pi_1: T^2\times S\lra T^2, \pi_2: T^2\times S\lra S.$$ Assume that $T^2=\mathbb{R}^2/\mathbb{Z}^2$ has coordinate $(x,y)$, then $\f{\p}{\p x}, \f{\p}{\p y}$ is a global frame on $T^2$. The tangent bundle of $X$ has a splitting $TX=TT^2\times TS$. Let $J_S$ be the complex structure on $S$ with local holomorphic coordinate $w$ and $h=h(w)$ be a smooth real nonconstant function on $S$. $h$ is pulled back by $\pi_2$ to be a function on $X$ (we still denote it by $h$ which is constant on $(x,y)$ direction). Define an almost complex structure on $X$ by

$$J(\f{\p}{\p x})=-h\f{\p}{\p x}+\f{\p}{\p y}, J(\f{\p}{\p y})=-(1+h^2)\f{\p}{\p x}+h\f{\p}{\p y}$$
$$J|_{TS}=J_S$$
Then $J^2=-id$ and $(TX)^{1,0}=<V, \f{\p}{\p w}>$, where $V=\f{\p}{\p x}+i(h\f{\p}{\p x}-\f{\p}{\p y})$. As $$[V,\f{\p}{\p w}]=-i\f{\p h}{\p w}\f{\p}{\p x}=-\f{i}{2}\f{\p h}{\p w}(V+\bar{V}),$$ $J$ is not integrable by Newlander-Nirenberg's theorem since $\f{\p h}{\p w}\neq 0$.\\
We have $J(dx)=-(hdx+(1+h^2)dy)$. Let $\a=dx+i(hdx+(1+h^2)dy)$. Then locally $$(T^*X)^{1,0}=<\a, dw> \  \text{and}\ \ \mathcal K_J^{\otimes m}=<(\a\wedge dw)^{\otimes m}>$$ for any $m\geq 1$. There is an embedding $$\pi_2^*: \Gamma(S, \mathcal K_S^{\otimes m})\lra \Gamma(X, \mathcal K_J^{\otimes m})$$ given by $\pi_2^*(\gamma)=(\a)^{\otimes m}\wedge \gamma$ for any $\gamma\in \Gamma(S, \mathcal K_S^{\otimes m})$.
Defining the $(0,1)$ form $\b=\f{-i(h+i)}{2(h-i)}\bar{\p} h$, we get

\begin{lem} \label{induce}
$\pi_2^*(\gamma)\in H^0(X, \mathcal K^{\otimes m})$ if and only if $\bar{\p} \gamma+m\b \wedge \gamma=0$.
\end{lem}

\begin{proof}
Assume locally that $\gamma=f(w)dw^{\otimes m}$. Then $\pi_2^*(\gamma)=(\a)^{\otimes m}\wedge \gamma=f(\a\wedge dw)^{\otimes m}$. We have
\begin{align*} \bar{\p}_J (\a\wedge dw)&=(d\a\wedge dw)^{2,1}\\ &=\f{i}{2}\f{\p h}{\p \bar{w}}(1-\f{2h(i+h)}{1+h^2})d\bar{w}\wedge \a\wedge dw\\ &=\f{-i(h+i)}{2(h-i)}\f{\p h}{\p \bar{w}}d\bar{w}\wedge \a\wedge dw.\end{align*}
 Let $b=\f{-i(h+i)}{2(h-i)}\f{\p h}{\p \bar{w}}$ and $\b=\f{-i(h+i)}{2(h-i)}\bar{\p} h$. As $f$ depends only on $w$, we have $$\bar{\p}_J(\pi_2^*(\gamma))= (\f{\p f}{\p \bar{w}}+mb f)d\bar{w}(\a\wedge dw)^{\otimes m}.$$ So $\bar{\p}_J(\pi_2^*(\gamma))=0$ is equivalent to $\f{\p f}{\p \bar{w}}+mb f=0$ which gives $\bar{\p} \gamma+m\b \wedge \gamma=0$.
\end{proof}

 Denote $$H_h^0(S, \mathcal K_S^{\otimes m})=\{\gamma\in \Gamma(S, \mathcal K_S^{\otimes m}), \bar{\p} \gamma+m\b \wedge \gamma=0\}.$$ When $h=0$, then $\b=0$ and the group is the ordinary holomorphic pluricanonical section group of $\mathcal K_S^{\otimes m}$.  From Lemma \ref{induce}, we get an injective map $$\pi_2^*:H_h^0(S, \mathcal K_S^{\otimes m})\lra H^0(X, \mathcal K^{\otimes m}).$$ We can compute the dimension of $H_h^0(S, \mathcal K_S^{\otimes m})$ explicitly when $m>1$. Notice that the operator $\bar{\p}_h=\bar{\p}+m\b\wedge\_$ satisfies the Leibniz rule and then gives a deformed holomorphic structure of $\mathcal K_S^{\otimes}$ as $\dim S=1$. For the holomorphic line bundle $(\mathcal K_S^{\otimes m}, \bar{\p}_h)$, $$\deg(\mathcal K_S^{\otimes m}, \bar{\p}_h)=\deg(\mathcal K_S^{\otimes m}, \bar{\p})=2m(g-1)$$ by the deformation invariance of $c_1$. Also,
 $$H^0(S, \mathcal K_S\otimes (m\mathcal K_S, \bar{\p}_h)^*)=0$$ for $m>1$ since the degree is negative.
 Applying the Riemann-Roch formula to $(\mathcal K_S^{\otimes m}, \bar{\p}_h)$, we have
 \begin{align} \dim H_h^0(S, \mathcal K_S^{\otimes m})&=\dim H_h^0(S, \mathcal K_S^{\otimes m})-\dim H^0(S, \mathcal K_S\otimes (\mathcal K_S^{\otimes m}, \bar{\p}_h)^*) \notag \\
 &=\deg(\mathcal K_S^{\otimes m}, \bar{\p}_h)-g+1 \notag \\
 &=(2m-1)(g-1) \label{RR} \end{align}
for $m>1$. When $m=1$, as  $\deg(\mathcal K_S\otimes (\mathcal K_S, \bar{\p}_h)^*)=0$, we have $\dim H^0(S, \mathcal K_S\otimes (\mathcal K_S, \bar{\p}_h)^*)\leq 1$. Applying the Riemann-Roch formula we obtain $$g-1\leq \dim H_h^0(S, \mathcal K_S)\leq g.$$

 Next, we show that $\pi_2^*$ is surjective.

\begin{lem} \label{pull}
For any $s\in H^0(X, \mathcal K^{\otimes m})$, $s=\pi_2^*(\gamma)$ for some $\gamma\in H_h^0(S, \mathcal K_S^{\otimes m})$.
\end{lem}

\begin{proof} We offer two different proofs. The first follows from direct calculation. The second proof applies the results of intersection theory built in \cite{Z2} which can be generalized to other cases (Theorem \ref{ellk=1} below).

Assume that $s=g(\a\wedge dw)^m$ locally. As $\bar{\p}\a=\beta\wedge\a$, we have $\bar{\p}_m s=(\bar{\p}g+mg\beta)(\a\wedge dw)^m$. So $\bar{\p}g+mg\beta=0$, which is equivalent to
\begin{align} \bar{V}(g)=0  \label{24} \\ (\f{\p g}{\p \bar{w}}+mb g)d\bar{w}=0 \end{align}
Using the same technique in example 6.1, from (\ref{24}) we get that $g$ is independent of $x, y$. So we can define $\gamma=g(dw)^{\otimes m}\in H_h^0(S, \mathcal K_S^{\otimes m})$, and $s=\pi_2^*(\gamma)$.\vspace{.2cm}

The second approach is more topological. Define a deformation $$J_t(\f{\p}{\p x})=-(th)\f{\p}{\p x}+\f{\p}{\p y}, J_t(\f{\p}{\p y})=-(1+t^2h^2)\f{\p}{\p x}+(th)\f{\p}{\p y}$$ $$J_t|_{TS}=J_S, 0\leq t\leq 1.$$ Then $J_1=J$ and $J_0$ is the product complex structure on $X$. By the homotopy invariance of the Chern classes, we have $c_1(\mathcal K_J)=c_1(\mathcal K_{J_0})=(2g-2)[T^2]$, where $[T^2]$ is the cohomology class of the fiber of $\pi_2$. Also, each fiber $T^2$ is a $J$-holomorphic curve by definition.

Let $z_0=(t_0,w_0)$ be any point in $X$ where $t_0\in T^2, w_0\in S$ and $s\in H^0(X, \mathcal K^{\otimes m})$ a nontrivial section. First assume that  $s(z_0)=0$. By Corollary \ref{holo}, $s$ induces a holomorphic map. Therefore, $s^{-1}(0)$ supports a pseudoholomorphic 1-subvariety in $X$ (Corollary 1.3 in \cite{Z2}). By the positive intersection of pseudoholomorphic curves (see \cite{Z2}), either $T^2\times \{w_0\}\subset s^{-1}(0)$ or $T^2\times \{w_0\}$ has positive intersection with $s^{-1}(0)$. As $$[s^{-1}(0)]=m\cdot c_1(\mathcal K_J)=m(2g-2)[T^2]$$ and $[T^2]\cdot [T^2]=0$, the latter case cannot be possible. So $s|_{T^2\times \{w_0\}}=0$.

Next, assume that $s(z_0)\neq 0$.
Denote $H^1_h(S, \mathcal K_S^{\otimes m}-\{w_0\})$ the first sheaf cohomology group of the tensor bundle $(\mathcal K_S^{\otimes m}, \bar{\p}_h)\otimes (-\{w_0\})$. By the Kodaira vanishing theorem, when $m>1$, $$H^1_h(S, \mathcal K_S^{\otimes m}-\{w_0\})=0$$ as $\deg(\mathcal K_S^{\otimes (m-1)}, \bar{\p}_h)\geq 2$. From the exact sequence $$0\lra  \mathcal K_S^{\otimes m}-\{w_0\}\lra  \mathcal K_S^{\otimes m}\lra  {\mathcal K_S^{\otimes m}}|_{w_0}\lra 0,$$ we get the exact sequence of cohomology groups (\cite{GH}) : $$0\lra H_h^0(S, \mathcal K_S^{\otimes m}-\{w_0\})\lra H^0_h(S, \mathcal K_S^{\otimes m})\lra {\mathcal K_S^{\otimes m}}|_{w_0}\lra H^1_h(S, \mathcal K_S^{\otimes m}-\{w_0\})=0.$$ Therefore,
there is a $\tilde{\gamma}\in H_h^0(S, \mathcal K_S^{\otimes m})$ such that $\tilde{\gamma}(w_0)\neq 0$ when $m>1$. Then $\pi_2^*(\tilde{\gamma})(z_0)\neq 0$. Since $s(z_0)\neq 0$, there is some $k\neq 0$ such that $(s-k\pi_2^*(\tilde{\gamma}))(z_0)=0.$ As $s-k\pi_2^*(\tilde{\gamma})\in H^0(X, \mathcal K^{\otimes m}),$ by the same argument as in the first case, $$(s-k\pi_2^*(\tilde{\gamma}))|_{T^2\times \{w_0\}}=0.$$ So $s=k\pi_2^*(\tilde{\gamma})$ on $T^2\times \{w_0\}$.

Therefore, in either case, $s$ is constant on the fiber of $\pi_2$. Then we can push down $s$ through $\pi_2$ to get a section $\gamma\in \Gamma(S, \mathcal K_S^{\otimes m})$ such that $s=\pi_2^*(\gamma)$. By Lemma \ref{induce}, $\gamma\in H_h^0(S, \mathcal K_S^{\otimes m})$.
\end{proof}

Combining Lemma \ref{induce}, Lemma \ref{pull} and (\ref{RR}), we have

\begin{prop} \label{4}
$\pi_2^*:H_h^0(S, \mathcal K_S^{\otimes m})\lra H^0(X,\mathcal K^{\otimes m})$ is an isomorphism. Therefore, $P_m(X,J)=\dim H_h^0(S, \mathcal K_S^{\otimes m})=(2m-1)(g-1)$ for $m>1$, $g-1\leq P_1(X,J)\leq g$ and $\kappa^J(X)=1$.
\end{prop}

To compute the irregularity of $X$, assume that $\tau \in H^0(X, \Omega_X)$. Locally write $\tau=g_1\a+g_2dw$. From $\bar{\p}\tau=0$ we get \begin{align} \bar{\p}g_1+g_1\beta=0 \label{26} \\ \bar{\p}g_2=0. \label{27}\end{align} From (\ref{26}) we get that \begin{align} \bar{V}(g_1)=0 \label{027}\\ \f{\p g_1}{\p\bar{w}}+bg_1=0.\label{28} \end{align} Then (\ref{027}) gives that $g_1$ is independent of $x,y$ as before. (\ref{28}) can be interpreted as follows. The $\bar{\p}_h=\bar{\p}+\beta\wedge\_$ also induces a deformed complex structure on the trivial line bundle. Define $H_h^0(S, \mathcal O)=\{g\in C^{\infty}(S),\bar{\p}_h g=0\}.$ Then (\ref{28}) is equivalent to $g_1\in H_h^0(S, \mathcal O)$. As $\deg \mathcal O=0$, we have $\dim H_h^0(S, \mathcal O)\leq 1$. From (\ref{27}) we get that $$\bar{V}g_2=0, \f{\p g_2}{\p\bar{w}}=0.$$ which implies that $g_2$ is constant. Therefore $\tau=g_1\a+cdw$, with $g_1\in H_h^0(S, \mathcal O)$. As $h^{1,0}(S)=g$, we obtain $$g\leq h^{1,0}(X)\leq g+1,$$
The case $h^{1,0}(X)= g+1$ corresponds to $\dim H_h^0(S, \mathcal O)=1$ which implies that $(\mathcal O, \bar{\p}_h)$ is holomorphic trivial. The case $h^{1,0}(X)= g$ corresponds to $\dim H_h^0(S, \mathcal O)=0$.\\

We can generalize the calculation to the case where $X$ admits a smooth pseudoholomorphic elliptic fibration.

\begin{thm}\label{ellk=1}
If $(X^4, J)$ admits a smooth pseudoholomorphic elliptic fibration over a Riemann surface of genus greater than $1$ with $J$ tamed, then $\kappa^J=1$.
\end{thm}
\begin{proof}
Let $\pi: X\rightarrow S$ be the pseudoholomorphic elliptic fibration. By \cite{T96}, the canonical class $K$ is represented by $J$-holomorphic $1$-subvariety $\Theta$. For the fiber class $T$, we have $T\cdot T=0$. Hence $K\cdot T=0$ by adjunction formula. By positivity of intersection, any component of $\Theta$ is contained in a fiber. Since each fiber is smooth, we have $K=bT$. On the other hand, any section of $\mathcal K_X$ pushed down to a section of $\mathcal K_S$ by integrate out the fiber. Hence $K=(2g-2)T$.

In other words, as complex bundles, $\pi^*(\mathcal K_S^{\otimes m})=\mathcal K_X^{\otimes m}$. We  notice that $\bar{\p}_m$ maps $\pi^*\Gamma(S, \mathcal K_S^{\otimes m})$ to $\pi^*\Gamma(S, \mathcal K_S^{\otimes m}\otimes T^*S)$. We thus denote $\bar{\p}_m\pi^*(f\gamma)=\pi^*(\bar{\p}_{\pi}(\gamma))$ where $\bar{\p}_{\pi}$ is an operator mapping $\Gamma(S, \mathcal K_S^{\otimes m})$ to $\Gamma(S, \mathcal K_S^{\otimes m}\otimes T^*S)$. Hence, for any smooth function $f$ on $S$, by the Leibniz rule of $\bar{\p}_m$, we have $$\pi^*(\bar{\p}_{\pi}(f\gamma))=\bar{\p}_m\pi^*(f\gamma)=\bar{\p}\pi^*f\wedge\pi^*\gamma+\pi^*f\cdot \bar{\p}_m\pi^*\gamma=\pi^*(\bar{\p}_{\pi}f\wedge\gamma+f\bar{\p}_{\pi}\gamma).$$
That is to say $\bar{\p}_{\pi}$ also satisfies the Leibniz rule and hence it is a pseudoholomorphic structure on $\mathcal K_S^{\otimes m}$. Since $S$ is a Riemann surface, it defines a holomorphic structure on it. To summarize, $\pi^*(\gamma)\in H^0(X, \mathcal K^{\otimes m})$ if and only if $\gamma\in H^0_{\pi}(S, \mathcal K_S^{\otimes m})$, where $H^0_{\pi}(S, \mathcal K_S^{\otimes m})=\{\gamma\in \Gamma(S, \mathcal K_S^{\otimes m}), \bar{\p}_{\pi} \gamma=0\}.$

Since any section $s\in H^0(X, \mathcal K^{\otimes m})$ would have zero locus a $J$-holomorphic $1$-subvariety in class $mK=m(2g-2)T$, Lemma \ref{pull} (or the argument in the first paragraph) still applies and we know any section $s\in H^0(X, \mathcal K^{\otimes m})$ is of the form $\pi^*(\gamma)$ for some $\gamma\in H^0_{\pi}(S, \mathcal K_S^{\otimes m})$.

Therefore, $P_m(X,J)=\dim H_{\pi}^0(S, \mathcal K_S^{\otimes m})=(2m-1)(g-1)$ for $m>1$ and $\kappa^J(X)=1$.
\end{proof}

We remark that the only place we use tameness is that it guarantees the existence of pseudoholomorphic $1$-subvariety in the (pluri)canonical class.

In fact, the examples in Section \ref{kt} (as well as Section \ref{T4}) are smooth pseudoholomorphic elliptic fibrations over $T^2$. In these cases, $(\mathcal K^{\otimes m}, \bar{\p}_{\pi})$ are holomorphic line bundle of degree $0$. These bundles are holomorphically trivial if and only if $P_m=1$.

With those $4$-manifolds with $\kappa^J=1$, we can construct more nonintegrable almost complex manifolds with large Kodaira dimensions. First, we derive the K\"unneth formula for pluricanonical sections of almost complex manifolds. For two almost complex manifolds $(X_1, J_1)$ and $(X_2, J_2)$, the product map $J_1\times J_2$ induces an almost complex structure on $X_1\times X_2$. We have
 \begin{prop}\label{Kun}
 $P_m(X_1\times X_2, J_1\times J_2)=P_m(X_1, J_1)P_m(X_2, J_2)$ for $m\geq 1$.
 \end{prop}

\begin{proof}

We apply the harmonic theory in section 3 to derive the formula, similar to the argument in the integrable case (see \cite{GH}). Let $$p_1: X_1\times X_2\lra X_1, p_2: X_1\times X_2\lra X_2$$ be the two projections. We have $\mathcal K_{X_1\times X_2}=p_1^*(\mathcal K_{X_1})\wedge p_2^*(\mathcal K_{X_2})$. Choose Hermitian metrics $g_1$ and $g_2$ on $X_1, X_2$ respectively. Then $g_1\times g_2$ gives a Hermitian metric on $X_1\times X_2$. A form $\phi\in \Gamma(X_1\times X_2, \mathcal K_{X_1\times X_2})$ is called decomposable if $\phi=p_1^*(\phi_1)\wedge p_2^*(\phi_2)$. Similar arguments as those in Page 104 in \cite{GH} show that the decomposable smooth forms are dense in the Hilbert space $L^2(X_1\times X_2, \mathcal K_{X_1\times X_2})$.

Denote $\Delta_{J_1}, \Delta_{J_2}$ the Laplacian operators associated to $\bar{\p}_{J_1}, \bar{\p}_{J_2}$ as given in (\ref{lap}). By the definition, they are both semi-positive operators. Let $\varphi_1, \varphi_2, \cdots,$ be the eigenforms of $\Delta_{J_1}$ in $\Gamma(X_1, \mathcal K_{X_1})$ with eigenvalues $\lambda_1, \lambda_2, \cdots$ and $\psi_1, \psi_2, \cdots,$ be the eigenforms of $\Delta_{J_2}$ in $\Gamma(X_2, \mathcal K_{X_2})$ with eigenvalues $\mu_1, \mu_2, \cdots$. Then $\lambda_i\geq 0, \mu_i\geq 0$ for any $i$. Let $\Delta_{J_1\times J_2}$ be the Laplacian operator associated to $J_1\times J_2$ and $g_1\times g_2$. From the definition, we directly get $\Delta_{J_1\times J_2}=\Delta_{J_1}+\Delta_{J_2}.$ Also, $$\Delta_{J_1\times J_2}(p_1^*(\varphi_i)\wedge p_2^*(\psi_j))=(\lambda_i+\mu_j)p_1^*(\varphi_i)\wedge p_2^*(\psi_j).$$ So we have $\Delta_{J_1\times J_2}(p_1^*(\varphi_i)\wedge p_2^*(\psi_j))=0$ if and only if $\lambda_i=\mu_j=0$. As $\{\varphi_i\}$, $\{\psi_i\}$ give Hilbert bases for $L^2(X_1, \mathcal K_{X_1})$ and $L^2(X_2, \mathcal K_{X_2})$ respectively, $\{p_1^*(\varphi_i)\wedge p_2^*(\psi_j)\}$ gives a Hilbert basis of $L^2(X_1\times X_2, \mathcal K_{X_1\times X_2})$ by the denseness of decomposable forms. Therefore, we get $Ker(\Delta_{J_1\times J_2})=<p_1^*(\varphi_i)\wedge p_2^*(\psi_j)>$ with $\lambda_i=\mu_j=0$. Namely,
$$H^0(X_1\times X_2,\mathcal K_{X_1\times X_2})=H^0(X_1, \mathcal K_{X_1})\otimes H^0(X_2, \mathcal K_{X_2}).$$
This shows that $P_1(X_1\times X_2, J_1\times J_2)=P_1(X_1, J_1)P_1(X_2, J_2)$. Similar argument gives that $$H^0(X_1\times X_2,\mathcal K_{X_1\times X_2}^{\otimes m})=H^0(X_1, \mathcal K_{X_1}^{\otimes m})\otimes H^0(X_2, \mathcal K_{X_2}^{\otimes m})$$ and $P_m(X_1\times X_2, J_1\times J_2)=P_m(X_1, J_1)P_m(X_2, J_2)$ for any $m>1$.
\end{proof}

From the definition of Kodaira dimension we have

\begin{cor}
$\kappa^{J_1\times J_2}(X_1\times X_2)=\kappa^{J_1}(X_1)+\kappa^{J_2}(X_2)$ for any two compact almost complex manifolds $(X_1, J_1), (X_2, J_2)$.
\end{cor}

\begin{thm}\label{niKod}
There are examples of compact $2n$-dimensional nonintegrable almost complex manifolds with Kodaira dimension lying among $\{-\infty, 0, 1, \cdots, n-1\}$ for $n\geq 2$.
\end{thm}

\begin{proof} By taking direct products of the Kodaira-Thurston surface with copies of two torus $T^2$, we can get compact $2n$-manifolds with nonintegrable almost complex structure and $\kappa^J=-\infty$ or $0$.

By taking direct products of the 4-manifold $X=T^2\times S$ as in Proposition \ref{4} with copies of  $2$-torus $T^2$ or the Riemann surface $\Sigma$ with $g>1$, we get compact $2n$-manifolds with nonintegrable almost complex structures and $\kappa^J=1, 2, \cdots, n-1$.
\end{proof}

\section{The six sphere}

By a result of Borel and Serre \cite{BS}, the only spheres which admit almost complex structures are $S^2$ and $S^6$. The standard way to construct an almost complex structure on $S^6$ is to use the cross product of $\mathbb R^7$ applying to the tangent space of $S^6$.  In this section, we will compute the Hodge numbers, the plurigenera and Kodaira dimension of the standard almost complex structure. Our method is to consider $S^6$ as a homogeneous space of the exceptional Lie group $G_2$ and apply an explicit real representation of the Lie algebra $\mathfrak g_2$.

First, we review some definitions following \cite{Br3}.
Let $e_1, e_2, \cdots,e_7$ be the standard basis of $\mathbb R^7$ and $e^1, e^2,\cdots, e^7$ be the dual basis. Denote $e^{ijk}$ the wedge product $e^i\wedge e^j\wedge e^k$ and define $$\Phi=e^{123}+e^{145}+e^{167}+e^{246}-e^{257}-e^{347}-e^{356}$$
Then $\Phi$ induces a unique bilinear mapping, the cross product: $\times : \mathbb R^7\times \mathbb R^7\lra \mathbb R^7$ by $(u\times v)\cdot w=\Phi(u,v,w)$, where $\cdot$ is the Euclidean metric on $\mathbb R^7$. It follows that $u\times v=-v\times u$ and \begin{align}\label{c1} (u\times v)\cdot u=0.\end{align} Also, further discussion (see \cite{Br3}) shows that \begin{align}\label{c2} u\times (u\times v)=(u\cdot v)u-(u\cdot u)v.\end{align} We remark that the cross product $\times$ differs from the cross product induced by Cayley's table of Octonion multiplication, though they are isomorphic. For example, here $e_1\times e_6=e_7$.

 The six sphere $S^6=\{u\in \mathbb R^7, u\cdot u=||u||=1\}.$ The tangent space at $u\in S^6$ is $T_uS^6=\{v\in \mathbb R^7| u\cdot v=0\}$. Let $J_u=u\times \_$ be the cross product operator of $u$. Then by (\ref{c1}),(\ref{c2}), $J_u(T_uS^6)\subset T_uS^6$ and $J_u^2=-id$ on $T_uS^6$. In particular, when $u=e_1$, we have
\begin{align}\notag &J_{e_1}(e_2)=e_3, \quad J_{e_1}(e_3)=-e_2, \quad J_{e_1}(e_4)=e_5,\\ &J_{e_1}(e_5)=-e_4, \quad J_{e_1}(e_6)=e_7, \quad J_{e_1}(e_7)=-e_6.\label{Je}\end{align} Let $\mathsf J=\{J_u, u\in S^6\}$. Then $\mathsf J$ gives an almost complex structure on $S^6$ which is the standard almost complex structure we consider. It is shown \cite{EF}\cite{EL} that $\mathsf J$ is not integrable since the Nijenhuis tensor of $\mathsf J$ is nowhere vanishing.

On the other side, denote \begin{align} \label{G2} G_2=\{g\in GL(7, \mathbb R)| g^*(\Phi)=\Phi\},\end{align} where $G_2$ is the simple Lie group of type $G_2$ which is compact, connected and simply connected with real dimension 14 (see \cite{Br}). $G_2$ preserves the inner product $\cdot$ and the cross product $\times$ and acts transitively on $S^6$. Let $G_2\times S^6\lra S^6$ be the transitive action and $p:G_2\lra S^6$ the induced map given by $p(g)=g(e_1)$. The map $p$ is a submersion with $p^{-1}(e_1)=\{g\in G_2| g(e_1)=e_1\}\cong SU(3)$. This makes $G_2$ into a principle right $SU(3)$ bundle over $S^6$.

Next, we give the explicit representation of the Lie algebra $\mathfrak g_2$ of $G_2$ and define a left invariant almost complex structure on it. Let $\epsilon_{ijk}$ be skew-symmetric unit indices such that $\Phi=\frac{1}{6}\epsilon_{ijk}e^{ijk}$. For example, $\epsilon_{123}=-\epsilon_{132}=\epsilon_{231}=1$. By the characterization in \cite{Br2} (section 2.5 there), a skew-symmetric matric $A=(a_{jk})$ is in $\mathfrak g_2$ if and only if $\sum_{j,k=1}^7\epsilon_{ijk}a_{jk}=0$ for all $1\leq i\leq 7$.
Let $\vec{x}=(x_1,x_2,x_3,x_4,x_5,x_6)\in \mathbb R^6$, $\vec{y}=(y_1,y_2,\cdots,y_8)\in \mathbb R^8$. Direct calculation gives that a general element $A$ in $\mathfrak g_2\subset gl(7,\mathbb R)$ has the following form
\begin{align}
A=\{\vec{x},\vec{y}\}:=
\begin{pmatrix} 0&x_1&-x_2&x_3&-x_4&x_5&-x_6\\
  -x_1&0&y_1&-x_6+y_4&x_5+y_3&x_4-y_6&-x_3-y_5\\
  x_2&-y_1&0&-y_3&y_4&y_5&-y_6\\
  -x_3&x_6-y_4&y_3&0&-y_1+y_2&-x_2-y_8&x_1-y_7\\
  x_4&-x_5-y_3&-y_4&y_1-y_2&0&y_7&-y_8\\
  -x_5&-x_4+y_6&-y_5&x_2+y_8&-y_7&0&-y_2\\
  x_6&x_3+y_5&y_6&-x_1+y_7&y_8&y_2&0\
\end{pmatrix}\notag
\end{align}
Here $\{\cdot,\cdot\}$ denotes an operation $\{\cdot,\cdot\}: \mathbb R^6\times \mathbb R^8\longrightarrow \mathfrak g_2$ whose definition is stated above. The above expression is chosen so that it suits our later discussion on $S^6$.

Denote $\vec{\alpha}_i, i=1,\cdots,6,$ the $i$-th unit vector in $\mathbb R^6$ and $\vec{\beta}_j, j=1,\cdots,8,$ the $j$-th unit vector in $\mathbb R^8$. Define
$f_i=\{\vec{\alpha}_i,\vec{0}\}, h_j=\{\vec{0},\vec{\beta}_j\}$. For example, $f_1=\{(1,0,0,0,0,0),\vec{0}\}, h_2=\{\vec{0},(0,1,0,0,0,0,0,0)\}$. Then $\{f_i, h_j;1\leq i\leq 6,1\leq j\leq 8\}$ forms a basis of $\mathfrak g_2$. The Lie brackets between $f_i$ and $h_j$ are computed in the appendix. Let $$\mathfrak m=span\{f_1,\cdots,f_6\},\hspace{.8cm} \mathfrak h=span\{h_1,\cdots,h_8\}.$$
Then $\mathfrak g_2=\mathfrak m\oplus \mathfrak h,\ [\mathfrak h,\mathfrak h]\subset \mathfrak h$ and $\mathfrak h\cong su(3).$ A Cartan subalgebra of $\mathfrak g_2$ is given by the span of $h_1, h_2$. The corresponding decomposition of $\mathfrak g_2$ into root spaces can be also calculated. For the projection $p:G_2\lra S^6$, we have $$\ker \ dp=\mathfrak h,\ \ \ \ dp(f_i)=(-1)^ie_{i+1}.$$
Define an almost complex structure $\tilde{J}$ on $\mathfrak g_2$ by \begin{align*}
&\tilde{J}(f_1)=-f_2, \hspace{.4cm} \tilde{J}(f_3)=-f_4, \hspace{.4cm}\tilde{J}(f_5)=-f_6,\\
 &\tilde{J}(h_1)=-h_2,\ \  \tilde{J}(h_3)=-h_4, \ \ \tilde{J}(h_5)=-h_6,\ \ \ \tilde{J}(h_7)=-h_8.\end{align*}
$\tilde{J}$ induces a left invariant almost complex structure on $G_2$ which is still denoted by $\tilde{J}$. By (\ref{Je}), the following holds at $1_{G_2}$, \begin{align} \label{pseudo} dp\circ \tilde{J}=\mathsf{J}\circ dp.\end{align} Since both $\tilde{J}$ and $\mathsf{J}$ are $G_2$-invariant, (\ref{pseudo}) holds globally on $G_2$. In other words, $p$ is a $(\tilde{J},\mathsf{J})$-pseudoholomorphic map.

With the construction of $\tilde{J}$, we prove the following

\begin{thm} \label{sphere}
For the standard almost complex structure $\mathsf J$ on $S^6$, $h^{1,0}=h^{2,0}=h^{2,3}=h^{1,3}=0$, $P_m(S^6, \mathsf J)=1$ for any $m\geq 1$ and $\kappa^{\mathsf J}=0$.
\end{thm}

We would like to thank several people, including Huijun Fan, Valentino Tosatti, Jiaping Wang and Bo Yang for encouraging us to proceed the calculation.

\begin{proof}

Compute the plurigenera $P_m(S^6, \mathsf J)$ first. Denote $(T^*S^6)^{1,0}$ the bundle of $(1,0)$ forms on $S^6$ and $p^*$ the pull back map of forms. As $p$ is $(\tilde{J},\mathsf{J})$-pseudoholomorphic, we have $p^*((T^*S^6)^{1,0})\subset (T^*G_2)^{1,0}=(\mathfrak g_2^*)^{1,0}$. $p^*$ is injective since $p$ is a submersion.
Denote $\{f^i, h^j\}$ the basis in $\mathfrak g_2^*$, dual to $\{f_i,h_j\}$. Then $(\mathfrak g_2^*)^{1,0}$ is generated by $\{\phi^1,\cdots,\phi^7\}$, where
 \begin{align*}&\phi^1=f^1-if^2, \ \ \ \phi^2=f^3-if^4,\ \ \ \phi^3=f^5-if^6, \\ & \phi^4=h^1-ih^2,\ \ \phi^5=h^3-ih^4,\ \ \phi^6=h^5-ih^6, \ \ \phi^7=h^7-ih^8.\end{align*}
As $[\mathfrak h,\mathfrak h]\subset \mathfrak h$, using the Lie brackets in the appendix, we get
\begin{align*}
df^1&=-f^2\wedge h^1-2f^3\wedge f^6+f^3\wedge h^4-2f^4\wedge f^5-f^4\wedge h^3-f^5\wedge h^6+f^6\wedge h^5,\\
df^2&=f^3\wedge h^3+f^4\wedge h^4-f^5\wedge h^5-f^6\wedge h^6+f^1\wedge h^1,\\
df^3&=-f^1\wedge h^4+2f^1\wedge f^6+f^2\wedge f^5-f^2\wedge h^3+f^4\wedge h^1-f^4\wedge h^2-f^5\wedge h^8+f^6\wedge h^7,\\
df^4&=f^1\wedge f^5+f^1\wedge h^3-f^2\wedge h^4-f^3\wedge h^1+f^3\wedge h^2-f^5\wedge h^7-f^6\wedge h^8,\\
df^5&=-f^1\wedge f^4+f^1\wedge h^6-f^2\wedge f^3+f^2\wedge h^5+f^3\wedge h^8+f^4\wedge h^7+ f^6\wedge h^2,\\
df^6&=-2f^1\wedge f^3-f^1\wedge h^5+f^2\wedge h^6-f^3\wedge h^7+f^4\wedge h^8-f^5\wedge h^2.
\end{align*}
Therefore, the definition of $\bar{\partial}$ gives \begin{align} \bar{\partial} \phi^1&=-\frac{i}{2}\phi^1\wedge \bar{\phi}^4-i\phi^2\wedge \bar{\phi}^5+i\phi^3\wedge \bar{\phi}^6,\notag \\
 \bar{\partial} \phi^2&=-\frac{i}{2}\phi^1\wedge \bar{\phi}^3-\frac{1-i}{2}\phi^2\wedge \bar{\phi}^4+i\phi^3\wedge \bar{\phi}^7,\label{s6d} \\
  \bar{\partial} \phi^3&=\frac{i}{2}\phi^1\wedge \bar{\phi}^2-\frac{i}{2}\phi^2\wedge \bar{\phi}^1+\frac{1}{2}\phi^3\wedge \bar{\phi}^4. \notag\end{align}
 Then $$\bar{\partial} (\phi^1\wedge \phi^2\wedge \phi^3)=\bar{\partial}\phi^1\wedge \phi^2\wedge \phi^3- \phi^1\wedge \bar{\partial}\phi^2\wedge \phi^3+\phi^1\wedge \phi^2\wedge \bar{\partial}\phi^3=0.$$ By the arguments in \cite{Br3} (equation (2.11) in \cite{Br3}), $\phi^1\wedge \phi^2\wedge \phi^3$ induces a nowhere-vanishing $G_2$-invariant $(3,0)$-form $\Phi$ on $S^6$. As $p$ is pseudoholomorphic and $p^*$ is injective, $\bar{\partial}\Phi=0$.

Assume $s\in H^0(S^6, \mathcal K_{\mathsf J})$, then $s=f\Phi$, where $f$ is a smooth function on $S^6$. From $\bar{\partial}s=0$ we get that $\bar{\partial}f=0$. Since $S^6$ is compact, the maximum principle gives that $f$ is a constant. Therefore, $P_1(S^6, \mathsf J)=h^{3,0}=1$ with $\Phi$ being a generator. Similarly, we get $P_m(S^6, \mathsf J)=1$ for $m\geq 2$, with $\Phi^m$ being a generator of $H^0(S^6, K_{\mathsf J}^{\otimes m})$. So $\kappa^{\mathsf J}=0$.

Next, we compute the Hodge numbers $h^{1,0}$ and $h^{2,0}$. Assume that $s\in H^{1,0}(S^6)$. Then $p^*s$ is in the span space of $\{\phi^1,\phi^2, \phi^3\}$, satisfying $\bar{\partial}(p^*s)=0$. Let $p^*s=k_1\phi^1+k_2\phi^2+k_3\phi^3$, where $k_i$ are smooth functions on $G_2$. From (\ref{s6d}) we get that
\begin{align}
\bar{\partial}k_3&=ik_1\bar{\phi}^6+ik_2\bar{\phi}^7+\frac{1}{2}k_3\bar{\phi}^4. \label{k13}
\end{align}
Let $X_i, 1\leq i\leq 7,$ be the dual complex vector of $\phi^i$. Namely, $X_1=\frac{1}{2}(f_1+if_2), \cdots, X_7=\frac{1}{2}(h_7+ih_8)$. From the Appendix, the following Lie brackets hold
\begin{align}
[\bar{X}_1,\bar{X}_2]=-iX_3+\frac{i}{2}\bar{X}_5, \quad [\bar{X}_3,\bar{X}_5]=\frac{i}{2}h_1,\quad [X_3,\bar{X}_3]=\frac{i}{2}h_2.\label{X11}
\end{align}
Equation (\ref{k13}) gives us that $$\bar{X}_1(k_3)=\bar{X}_2(k_3)=\bar{X}_3(k_3)=\bar{X}_5(k_3)=0.$$
From (\ref{X11}), we have $X_3(k_3)=0$ and $h_1(k_3)=0$. Then by the last relation in (\ref{X11}), $h_2(k_3)=0$. So $\bar{X}_4(k_3)=0$. Evaluate $\bar{X}_4$ to (\ref{k13}) we get that $k_3=0$. Then (\ref{k13}) directly gives that $k_1=k_2=0$. Therefore, $p^*s=0$. As $p^*$ is injective, we get $s=0$. Hence, $H^{1,0}(S^6)=0$ and $h^{1,0}=0$.

To calculate $h^{2,0}$, assume that $\sigma\in H^{2,0}(S^6)$. Then $p^*\sigma$ satisfies $\bar{\partial}(p^*\sigma)=0$. Let $p^*\sigma=l_1\phi^1\wedge\phi^2+l_2\phi^2\wedge\phi^3+l_3\phi^3\wedge\phi^1$, where $l_i$ are smooth functions on $G_2$. From (\ref{s6d}) we get
\begin{align*} \bar{\partial} (\phi^1\wedge \phi^2)&=\frac{1}{2}\phi^1\wedge\phi^2\wedge \bar{\phi}^4+i\phi^2\wedge \phi^3\wedge\bar{\phi}^6-i\phi^1\wedge \phi^3\wedge \bar{\phi}^7,\notag \\
 \bar{\partial} (\phi^2\wedge \phi^3)&=\frac{i}{2}\phi^1\wedge \phi^3\wedge\bar{\phi}^3-\frac{i}{2}\phi^2\wedge\phi^3\wedge \bar{\phi}^4+\frac{i}{2}\phi^1\wedge \phi^2\wedge \bar{\phi}^2,\notag \\
  \bar{\partial} (\phi^3\wedge \phi^1)&=-\frac{i}{2}\phi^1\wedge \phi^2\wedge\bar{\phi}^1+\frac{1-i}{2}\phi^1\wedge\phi^3\wedge \bar{\phi}^4-i\phi^2\wedge \phi^3\wedge \bar{\phi}^5. \end{align*}
So $\bar{\partial}(p^*\sigma)=0$ gives that
\begin{align}
\bar{\partial}l_2&=-il_1\bar{\phi}^6+\frac{i}{2}l_2\bar{\phi}^4+il_3\bar{\phi}^5. \label{l13}
\end{align}
Then $\bar{X}_1(l_2)=\bar{X}_2(l_2)=\bar{X}_3(l_2)=\bar{X}_7(l_2)=0$. By the Appendix, the following hold
\begin{align} &[\bar{X}_1,\bar{X}_7]=-\frac{i}{2}h_2,\hspace{1cm} [\bar{X}_2,\bar{X}_7]=i\bar{X}_6, \hspace{1cm} [\bar{X}_2,\bar{X}_6]=\frac{i}{2}(h_1+h_2).\label{h20} \end{align}
From (\ref{h20}), we have $h_1(l_2)=h_2(l_2)=0$. Then $\bar{X}_4(l_2)=0$. Putting it back to (\ref{l13}), we derive $l_1=l_2=l_3=0$. So $p^*\sigma=0$. By the injectivity, $\sigma=0$. Therefore, $H^{2,0}(S^6)=0$ and $h^{2,0}=0$.

By the Serre duality (Proposition \ref{Serre}), we have $h^{1,3}(S^6)=h^{2,3}(S^6)=0$.
\end{proof}

On the other hand, for a hypothetical complex structure on $S^6$, $P_1=h^{3,0}=0$. The key point is a $\bar{\p}$-closed $(3, 0)$ form is also $d$-closed on a complex $3$-fold.

\section{Appendix}

Direct calculation gives the following Lie brackets of $\mathfrak g_2$:

\begin{align*}&[f_1,f_2]=h_1+h_2,\hspace{.5cm}[f_1,f_3]=2f_6,\hspace{.5cm}[f_1,f_4]=f_5,\hspace{.5cm} [f_1,f_5]=-f_4,\\
&[f_1,f_6]=-2f_3,\hspace{.5cm}[f_1,h_1]=-(f_2-h_8),\hspace{.5cm}[f_1,h_2]=-h_8,\\
&[f_1,h_3]=-(f_4+h_6),\hspace{.5cm}[f_1,h_4]=f_3,\hspace{.5cm} [f_1,h_5]=f_6,\\
&[f_1,h_6]=-f_5+h_3,\hspace{.5cm}[f_1,h_7]=0,\hspace{.5cm}[f_1,h_8]=h_2\\
&[f_2,f_3]=f_5-h_3,\hspace{.5cm}[f_2,f_4]=-h_4,\hspace{.5cm}[f_2,f_5]=-f_3+h_5,\\
&[f_2,f_6]=h_6,\hspace{.5cm}[f_2,h_1]=f_1+h_7,\hspace{.5cm}[f_2,h_2]=-h_7,\\
&[f_2,h_3]=f_3-h_5,\hspace{.5cm}[f_2,h_4]=f_4,\hspace{.5cm}[f_2,h_5]=-f_5+h_3,\\
&[f_2,h_6]=-f_6,\hspace{.5cm}[f_2,h_7]=h_2,\hspace{.5cm}[f_2,h_8]=0\\
&[f_3,f_4]=h_2,\hspace{.5cm}[f_3,f_5]=h_8,\hspace{.5cm}[f_3,f_6]=2f_1,\hspace{.5cm}[f_3,h_1]=f_4+h_6,\\
&[f_3,h_2]=-f_4,\hspace{.5cm}[f_3,h_3]=-(f_2-h_8),\hspace{.5cm}[f_3,h_4]=-f_1,\hspace{.5cm}[f_3,h_5]=0,\\
&[f_3,h_6]=-(h_1+h_2),\hspace{.5cm}[f_3,h_7]=f_6,\hspace{.5cm}[f_3,h_8]=-f_5,
\end{align*}
\begin{align*}
&[f_4,f_5]=2(f_1+h_7),\hspace{.5cm}[f_4,f_6]=h_8,\hspace{.5cm}[f_4,h_1]=h_5-f_3,\\
&[f_4,h_2]=f_3,\hspace{.5cm}[f_4,h_3]=f_1+h_7,\hspace{.5cm}[f_4,h_4]=-f_2,\\
&[f_4,h_5]=-(h_1+h_2),\hspace{.5cm}[f_4,h_6]=0,\hspace{.5cm}[f_4,h_7]=-f_5,\hspace{.5cm}[f_4,h_8]=-f_6,\\
&[f_5,f_6]=-h_2,\hspace{.5cm}[f_5,h_1]=h_4,\hspace{.5cm}[f_5,h_2]=f_6,\\
&[f_5,h_3]=0,\hspace{.5cm}[f_5,h_4]=-h_1,\hspace{.5cm}[f_5,h_5]=f_2-h_8,\\
&[f_5,h_6]=f_1+h_7, \hspace{.5cm}[f_5,h_7]=f_4,\hspace{.5cm}[f_5,h_8]=f_3,\\
&[f_6,h_1]=h_3,\hspace{.5cm}[f_6,h_2]=-f_5,\hspace{.5cm}[f_6,h_3]=-h_1,\hspace{.5cm}[f_6,h_4]=0,\\
&[f_6,h_5]=-f_1,\hspace{.5cm}[f_6,h_6]=f_2,\hspace{.5cm}[f_6,h_7]=-f_3,\hspace{.5cm}[f_6,h_8]=f_4,\\
&[h_1,h_2]=0, \hspace{.5cm} [h_1,h_3]=-2h_4,\hspace{.5cm} [h_2,h_3]=h_4,\hspace{.5cm} [h_1,h_4]=2h_3,\\
&[h_2,h_4]=-h_3,\hspace{.5cm}[h_1,h_5]=[h_2,h_5]=-h_6, \hspace{.5cm} [h_1,h_6]=[h_2,h_6]=h_5,\\
 &[h_1,h_7]=h_8,\hspace{.5cm} [h_2,h_7]=-2h_8,\hspace{.5cm} [h_1,h_8]=-h_7, \hspace{.5cm}[h_2,h_8]=2h_7.
\end{align*}

\end{document}